\documentclass[11pt,oneside]{article}
\usepackage{amsmath, amsthm, amssymb, color, amsrefs, authblk,arydshln}
\usepackage{mathrsfs,mathtools,enumerate,hyperref}
\usepackage[a4paper, total={6in, 9.5in}]{geometry}

\newtheorem{thm}{Theorem}[section]
\newtheorem*{thm*}{Theorem}
\newtheorem{prop}[thm]{Proposition}
\newtheorem{lem}[thm]{Lemma}
\newtheorem{cor}[thm]{Corollary}

\theoremstyle{definition}
\newtheorem{definition}[thm]{Definition}
\newtheorem{example}[thm]{Example}
\newtheorem*{ex*}{Example}

\theoremstyle{remark}
\newtheorem{remark}[thm]{Remark}
\newtheorem*{remark*}{Remark}

\DeclareMathOperator\vol{vol}

\DeclareMathOperator\GL{GL}

\DeclareMathOperator\U{U}
\DeclareMathOperator\SU{SU}

\DeclareMathOperator\Spin{Spin}

\DeclareMathOperator\Hess{Hess}

\DeclareMathOperator\Vol{Vol}

\def\d{{\rm d}}
\def\i{\mathrm{i}}
\def\j{\mathrm{j}}
\def\w{\wedge}
\def\ph{\varphi}
\DeclareMathOperator\tr{tr}
\def\C{\mathbb{C}}

\def\R{\mathbb{R}}

\def\Z{\mathbb{Z}}

\def\Ker{\mathrm{ker}}
\def\Im{\mathrm{im}}

\DeclareMathOperator\W{\wedge}
\DeclareMathOperator\ep{\varepsilon}

\begin{document}

\title{Pluripotential theory and special holonomy:\\ Hodge decompositions and $\partial\bar\partial$ lemmata}

\author{Tommaso Pacini and Alberto Raffero}
\affil{Department of Mathematics, University of Torino \\ via Carlo Alberto 10, 10123 Torino, Italy \\ tommaso.pacini@unito.it, alberto.raffero@unito.it}
\date{}
\maketitle

\begin{abstract}
Let $M$ be a compact torsion-free $G_2$ 7-manifold or Calabi-Yau 6-manifold. We prove Hodge decomposition theorems for the $dd^\phi$ operators, 
introduced by Harvey and Lawson \cite{HL:intropotential}, which generalize the $i\partial\bar\partial$ operator used in classical pluripotential theory. 
We then obtain analogues of the $\partial\bar\partial$ lemma in this context. 
We formalize this by defining cohomology spaces analogous to Bott-Chern cohomology and we relate them to harmonic forms on $M$. 

In the $G_2$ case we provide a geometric interpretation of the corresponding cohomology classes in terms of coassociative submanifolds and gerbes: 
this is analogous to the classical interpretation of Bott-Chern cohomology classes in terms of divisors and holomorphic line bundles.
\end{abstract}

\section{Introduction}\label{s:Intro}

The main goal of this paper is to initiate the study of Hodge theory for certain natural geometric operators which appear in the context of calibrated geometry \cite{HL:intropotential}. More specifically, we are interested in manifolds with special holonomy. 

Our results should be compared to the classical Hodge theory for K\"ahler manifolds which, in particular through the $\partial\bar\partial$ lemma, lies at the heart of many contemporary problems in K\"ahler geometry. 

Although Harvey and Lawson manage to define and study their operators within a fairly unified framework, our results require an in-depth analysis of their interactions 
with the representation theory underlying each specific situation. We thus focus on two concrete classes of compact manifolds with special holonomy: torsion-free $G_2$ 7-manifolds and Calabi-Yau 6-manifolds. 
Our main results are theorems \ref{thm:G2Hodge}, \ref{thm:SU3Hodge} 
and corollaries \ref{cor:G2globallemma}, \ref{cor:G2coh}, \ref{cor:SU3globallemma}, \ref{cor:SU3coh}. Theorem \ref{ThmSplittingKahler} and Corollary \ref{cor:ddomega} provide a basis for the comparison with K\"ahler manifolds, also making it clear that the same techniques may be used to study many other situations.

These results and techniques are of an algebraic and functional analytic nature. Section \ref{s:gerbes}, in particular Corollary \ref{cor:G2BottChern}, complements this by providing a geometric perspective, thus achieving a more appropriate balance between analytic and geometric aspects of this topic.

We can explain our motivation, goals and results as follows.

\paragraph{Hodge decompositions.}
Let $M$ be a smooth manifold. Consider a linear differential operator $P:\Lambda^0(E)\rightarrow\Lambda^0(F)$ of order $m$, with smooth coefficients, 
between smooth sections of vector bundles $E,F$ on $M$. We shall use the term \textit{Hodge decomposition} to refer to a splitting of $\Lambda^0(F)$ of the following type:
\[
\Lambda^0(F)=\Im(P)\oplus\ker(Q),
\]
where $Q$ is an auxiliary operator. Typical, desirable, properties are:
\begin{enumerate}[(i)]
\item the decomposition should be orthogonal with respect to some $L^2$-type metric on $\Lambda^0(F)$;
\item $\ker(Q)$ should be finite-dimensional.
\end{enumerate}

It is sometimes possible to determine a priori a subspace $V\leq\Lambda^0(F)$ such that $\Im(P)\leq V$. In this case we will look for a splitting of $V$. If this subspace is defined via algebraic conditions, it will usually correspond to a subbundle $F'\leq F$: this case is similar to the one above. If instead it is defined by differential conditions, it will be convenient to incorporate them into $Q$. 
Appendix \ref{Sect:AnalyticTools} contains a detailed description of the analytic techniques used throughout the paper. Remarks \ref{rem:alternativeKahler}, \ref{rem:alternativeG2} discuss a possible alternative technique, based on the notion of elliptic complexes \cite{AtBo}.

\medskip 

\noindent\textit{Example. }
Assume $(M,g)$ is a compact oriented Riemannian manifold. Riemannian Hodge theory then studies the following cases: 
\begin{itemize}
\item $P$ is the standard Hodge Laplacian on differential forms
$$\Delta\coloneqq dd^*+d^*d:\Lambda^k(M)\rightarrow\Lambda^k(M).$$
Here, the theory provides a decomposition $\Lambda^k(M)=\Im(\Delta)\oplus\ker(\Delta)$.
\item $P$ is the standard exterior derivative on differential forms
$$d:\Lambda^{k-1}(M)\rightarrow\Ker(d)\leq\Lambda^k(M).$$
Here, the theory provides a decomposition $\ker(d)=\Im(d)\oplus\ker(\Delta)$.
\end{itemize}
In the first case, the decomposition serves to characterize $\Im(\Delta)$ as the orthogonal complement of the space of harmonic forms.

The second case introduces several interesting twists. First, $\Im(d)$ is contained in a subspace defined by a differential condition. 
Second, it relates two different operators $d$, $\Delta$. 
Third, the matter of characterizing $\Im(d)$ can be reformulated as follows: when is a given $k$-form $\alpha$ exact?
An obvious necessary condition is $d\alpha=0$. This is a differential condition, easy to check. 
A complete answer to the question, however, requires integration. Equivalently, it requires solving the global PDE $d\beta=\alpha$: 
this is usually not easy. The decomposition related to $P \coloneqq d$ shows however that the two conditions almost coincide: 
the space of exact forms coincides with the space of closed forms up to finite dimensions. 

This state of affairs is formalized in terms of the de Rham cohomology spaces $H^k(M;\R)\coloneqq \ker(d)/\Im(d)$. 
Hodge theory shows that (i) these spaces are finite-dimensional, (ii) they can be represented via harmonic forms. 
The fact that de Rham cohomology is isomorphic to singular cohomology then provides a deep link between analysis and topology.

\paragraph{Geometrically defined operators.}
Let $(M,J)$ be a complex manifold. The group $\GL(n,\C)$ acts on the bundle of complex $r$-forms on $M$ and determines its splitting into irreducible components $\Lambda^{p,q}(M)$. In turn, this defines a decomposition of the exterior derivative $d$ into operators $\partial$, $\bar\partial$. Given a compatible metric, 
Hermitian Hodge theory provides Hodge decompositions for these operators and the corresponding Laplacian operators. 
These results are formalized in terms of the Dolbeault cohomology spaces $H^{p,q}(M)\coloneqq \ker(\bar\partial)/\Im(\bar\partial)$.

Other groups lead to other geometric structures and other decompositions of $d$.
Such structures appear, most prominently, in the context of special holonomy. 
The most classical example is the group $\U(n)$, leading to the notion of K\"ahler manifolds. 
Here, the splitting of forms and the decomposition of $d$ are studied via the Lefschetz operator which determines yet another splitting, leading to the well-known K\"ahler Hodge theory.

As mentioned, we shall focus on two other cases: 
(i) The group $G_2$, leading to the notion of torsion-free $G_2$ manifolds. 
Here, the decomposition of $d$ was worked out by Bryant and Harvey \cite{Bryant}.  
(ii) The group $\SU(3)$, leading to the notion of Calabi-Yau 6-manifolds. In Section \ref{Sec:CY6} we work out the analogous details in this case. 

Understanding these decompositions is a vital step towards our main goal, which is to study certain other geometrically defined operators, discussed below. In particular, the $\SU(3)$-induced decomposition of $d$ is fundamental for our Theorem \ref{thm:SU3Hodge}, but is also of independent interest.

\paragraph{Classical pluripotential theory.}The composition of $\partial,\bar\partial$ defines the operator $\partial\bar\partial$. This operator (more precisely: the corresponding real operator $i\partial\bar\partial$) provides the foundation for pluripotential theory. 

Let us focus on two fundamental results concerning this operator, cumulatively known as the $\partial\bar\partial$ lemmata. 
One is local, and applies to any complex manifold. The other is global, and applies only to certain subclasses of complex manifolds: 
most prominently, compact K\"ahler manifolds. These results express deep relationships between the operators $d, \partial, \bar\partial$ 
and $\partial\bar\partial$. We shall review this theory in Section \ref{s:dedebar}.

Once again, the situation is formalized in terms of a cohomology: the Bott-Chern cohomology, defined as $H^{p,q}_{BC}(M)\coloneqq(\ker(\partial)\cap\ker(\bar\partial))/\Im(\partial\bar\partial)$. The global $\partial\bar\partial$ lemma provides an isomorphism $H^{p,q}_{BC}(M)\simeq H^{p,q}(M)$.

These notions are of fundamental importance in geometry because $H^{1,1}(M)$ and  $H^{1,1}_{BC}(M)$ are the natural locus for many fundamental 
objects in K\"ahler geometry, such as the K\"ahler cone and the first Chern class for vector bundles. $H^{1,1}_{BC}(M)$ however concerns functions. 
The isomorphism $H^{1,1}(M)\simeq H^{1,1}_{BC}(M)$ thus opens a door between complex geometry and function theory. 
This has had a profound impact on modern K\"ahler geometry, for example in the search for special metrics, or in the theory of the K\"ahler-Ricci flow. 
In particular, it shows how to reduce certain systems of equations to equations for a single potential function. 

It is an interesting fact that these cohomologies and isomorphisms, even though they encode so much K\"ahler information, 
rely only on the complex structure.

The most efficient, thus most common, way of proving the global $\partial\bar\partial$ lemma is as a by-product of K\"ahler Hodge theory. This also serves to indicate interesting relationships between these operators and the Hodge Laplacian. 
For the purposes of this paper we will instead find it useful to emphasize a different proof, 
via an ad hoc Hodge decomposition for the operator $P\coloneqq 2i\partial\bar\partial$.
This is Theorem \ref{ThmSplittingKahler} and Corollary \ref{cor:ddomega}.

\paragraph{Cohomology and harmonic forms.}
In all the above situations, the Hodge Laplacian is known to preserve the splitting of forms induced by the group action. 
In turn, this induces a splitting of harmonic forms. Sometimes these spaces have a cohomological interpretation: 
\begin{itemize}
\item In the Riemannian context, the space of harmonic forms is isomorphic to de Rham cohomology.
\item In the K\"ahler context, the space of harmonic forms of type $(p,q)$ defined by the $\GL(n,\C)$-action is isomorphic to Dolbeault cohomology 
(alternatively, to Bott-Chern cohomology).
\item In the context of compact oriented 4-manifolds, one can build a cohomology space isomorphic to the space of harmonic self-dual forms.
\end{itemize}
On the other hand, to our knowledge there is no known cohomology corresponding to the spaces of harmonic forms induced by the 
$\U(n)$, $G_2$ or $\SU(3)$ action. 
Corollaries \ref{cor:G2coh}, \ref{cor:SU3coh} partially fill this gap: 
they show that some of these spaces are distinguished by the fact that they admit a cohomological interpretation.

\paragraph{{Pluripotential theory and calibrated geometry.}}In \cite{HL:calibratedgeometry} it is shown that K\"ahler geometry is a special case of a more general theory, known as \textit{calibrated geometry}. Here, one considers a Riemannian manifold $(M,g)$ endowed with a calibration $\phi$, namely a closed $k$-form $\phi$ 
such that $\phi|_\pi \leq \mathrm{vol}_\pi$, for all oriented tangent $k$-planes $\pi$ on $M.$  
In particular, Harvey and Lawson showed that calibrated geometry is the natural framework within which to study special features of manifolds with special holonomy.
 
Harvey and Lawson's later papers (see e.g.~\cite{HL:intropotential}), in the early 2000's, showed the existence of a geometrically defined $dd^\phi$ (or $\phi$-Hessian) operator on any calibrated manifold, and developed a corresponding pluripotential theory. 
In the case of K\"ahler manifolds, $dd^\phi=2i\partial\bar\partial$ so their theory extends the classical (complex) pluripotential theory mentioned above. The most basic take-away message from their papers is that all such theories are, from the analytic viewpoint, as rich as the classical theory. 

\paragraph{Summary of results.}
The above discussion lays out the rich, complex, {tapes}{try} of geometry and analysis which motivates and leads to the results in this paper. 
We can now summarize our results as follows.

Theorem \ref{ThmSplittingKahler} and Corollary \ref{cor:ddomega} provide a reformulation of the classical $\partial\bar\partial$ lemma, ideally suited to emphasize the uniform role of Hodge decompositions across the spectrum of manifolds with special holonomy.

Our main theorems \ref{thm:G2Hodge},\ref{thm:SU3Hodge} concern the operator $P\coloneqq dd^\phi$. These theorems are exactly analogous to Theorem \ref{ThmSplittingKahler}, after replacing the $\GL(n,\C)$ decomposition of $d$ with the analogous $G_2$, $\SU(3)$ decompositions mentioned above. In both cases, a key step consists in comparing the $dd^\phi$ operator with the corresponding decomposition of $d$, in order to detect the correct subspace of $\Lambda^0(F)$ within which to set up the Hodge decomposition.

As a by-product of our Hodge decomposition theorems, we obtain natural analogues of the $\partial\bar\partial$ lemma for manifolds with special holonomy. These results appear as corollaries \ref{cor:G2globallemma}, \ref{cor:SU3globallemma}. More precisely: we provide analogues of the global $\partial\bar\partial$ lemma, as stated in Corollary \ref{cor:ddomega}. In the $G_2$ context we also provide a counter-example to the local lemma.

We formalize these $\partial\bar\partial$ lemma analogues by defining corresponding cohomology spaces, analogous to Bott-Chern cohomology. We shall refer to these new cohomology spaces as the \textit{$dd^\phi$-cohomology}. Corollaries \ref{cor:G2coh}, \ref{cor:SU3coh} show that they are isomorphic to certain spaces of harmonic forms defined by the group action. 

Recall that Bott-Chern cohomology was originally introduced \cite{BottChern} in order to formalize the geometry of Chern connections on a holomorphic line bundle. In Section \ref{s:gerbes} we provide an analogous interpretation, in the $G_2$ context, of our $dd^\phi$-cohomology. This relies on the language of gerbes, thus reinforcing their potential role in the context of calibrated geometry \cite{HitchinGerbes}, \cite{GoncaloGerbes}. The main result is Corollary \ref{cor:G2BottChern}, which should be compared to Corollary \ref{cor:BottChern}.

\paragraph{Conclusions.}

Our results reinforce the analogies between the classical complex pluripotential theory and the extended theory defined by Harvey and Lawson. 
Although still firmly analytic, they shift the attention, however, towards those aspects which are most relevant to differential geometry, laying the foundations for two conceptually new, intriguing, questions:

\noindent 
(i) To what extent do our $dd^\phi$-cohomology spaces and the corresponding spaces of harmonic forms define a natural locus within which to develop $G_2$ and $\SU(3)$ geometry, beyond our current Section \ref{s:gerbes}? For example, it would be interesting to investigate possible connections to any $G_2$-analogue of the well-known Calabi conjecture in K\"ahler geometry, see \cite{SK_what}. This is work in progress.

\noindent
(ii) How to define higher-degree analogues of the Harvey-Lawson operator? This, and the corresponding decompositions, are also work in progress, potentially related to homotopy theory and issues of formality in the sense of \cite{DGMS}.

As mentioned above, in the K\"ahler case the relevant cohomology spaces rely only on the complex structure, rather than also on the K\"ahler form. The fact that these structures play separate roles is a distinguishing feature of K\"ahler geometry, which does not hold for other geometries with special holonomy. In our case, for example, the $dd^\phi$-cohomology uses the full $G_2$ or $\SU(3)$ structure. It remains to be seen what consequences this might have on any geometric developments.

It is also worth noticing that, in the $G_2$ case, our characterization of $\Im(dd^\phi)$ is equivalent to a characterization of the space of symmetric 
2-tensors of Hessian type. This is a result of independent interest. 
The corresponding problem in the general context of Riemannian geometry is, to our knowledge, still open (but see \cite{Bryant2,BE}).

Other recent work which uses cohomology to encode certain aspects of $G_2$ geometry includes \cite{CKT,KLS}. Their motivation and cohomology spaces are largely unrelated to ours, but perhaps related to our question (ii), above.

\bigskip

\noindent\textit{Acknowledgements: }
It is clear that this work rests upon the foundational papers of Harvey and Lawson. 
We are happy to thank Sandro Coriasco and J\"org Seiler for interesting conversations and Jason Lotay and Riccardo Piovani for comments on a first draft. 
We also wish to thank the anonymous referee for their very careful reading of the paper and for detecting an error in the first draft. This work was supported by the project PRIN 2022  
``Real and Complex Manifolds: Geometry and Holomorphic Dynamics'' and by GNSAGA of INdAM.  
 
A few historical remarks may also be in order. The global $\partial\bar\partial$ lemma for compact K\"ahler manifolds is due to Hodge. It appears in the (apparently largely unknown) paper \cite{Hodge}. A version of this result specific to $(1,1)$ forms also appears in \cite{Calabi}, equation (7). The modern, most common reference is \cite{DGMS}. The fact that the Laplacian preserves types and the relationship between the $\U(n)$- and Lefschetz splittings of forms appear in \cite{Chern}. 

Finally: the recent paper \cite{DS} extends our study to another class of manifolds with special holonomy, $\Spin(7)$-manifolds, finding analogous results.

\section{The $\partial\bar\partial$ lemmata and  Bott-Chern cohomology}\label{s:dedebar}

The goal of this section is to review two classical results in complex geometry known as the $\partial\bar\partial$ lemmata. 
One is local, and holds on any complex manifold. The other is global, and requires extra assumptions. 
The most typical setting is that of compact K\"ahler manifolds. In this case the lemma can be viewed as a corollary of K\"ahler Hodge theory.

\subsection{The decomposition of $d$.}

Let $(M,J)$ be a $n$-dimensional complex manifold. Recall that the spaces of complex-valued forms decompose into $\GL(n,\C)$-irreducible 
subspaces as follows:
\[
\Lambda^r(M,\C)=\bigoplus_{p+q=r}\Lambda^{p,q}(M).
\]
If we restrict the exterior differential $d$ to any $\Lambda^{p,q}(M)$, then project its image, we obtain the operators
\[
\partial:\Lambda^{p,q}(M)\rightarrow\Lambda^{p+1,q}(M), \ \ \bar\partial:\Lambda^{p,q}(M)\rightarrow\Lambda^{p,q+1}(M),
\]
so that $d=\partial + \bar\partial$.
The condition $d^2=0$ implies the conditions
\[
\partial^2=0,\quad \bar\partial^2=0, \quad \partial\bar\partial=-\bar\partial\partial.
\]
Complex conjugation sends forms of type $(p,q)$ to forms of type $(q,p)$. On the other hand, real forms are conjugation-invariant. 
It follows that those of pure type must be of type $(p,p)$. We are particularly interested in the case $p=1$. In this case, the operator 
\[
i\partial\bar\partial:\Lambda^0(M)\rightarrow\Lambda_\R^{1,1}(M)
\]
sends real-valued functions to real 2-forms of type $(1,1)$.

Now assume $M$ is K\"ahler. Let $\omega$ denote the K\"ahler form. Then there exists a refined decomposition of complex-valued forms into 
$\U(n)$-irreducible subspaces. It can be studied in terms of the Lefschetz operator
\[
L_\omega:\Lambda^r(M)\rightarrow\Lambda^{r+2}(M), \quad \alpha\mapsto\omega\wedge\alpha.
\]
For example, this leads to the orthogonal decomposition 
\[
\Lambda_\R^{1,1}(M)=\Lambda^0(M)\cdot\omega\oplus\Lambda^{1,1}_0(M),
\] 
which we will use below.

The Lefschetz operator is a real algebraic map, which modifies the degree of a differential form. 
Its interactions with the differential operators $\partial,\bar\partial$ are codified by the so-called K\"ahler identities, 
which lie at the heart of K\"ahler Hodge theory.

\begin{remark}
We will encounter analogous decompositions, algebraic operators and identities below, in the context of other geometries.
\end{remark}

\subsection{Cohomology.} 

On any complex manifold, the operators $\partial$ and $\bar\partial$ allow one to define several differential complexes, thus cohomology groups. 
We shall be particularly interested in the following:
\begin{enumerate}[$\bullet$]
\item The complex 
\[
\dots\rightarrow\Lambda^{p,q-1}(M) \xrightarrow[]{\bar\partial}  
\Lambda^{p,q}(M) \xrightarrow[]{\bar\partial}
\Lambda^{p,q+1}(M)\rightarrow\dots
\]
defines the Dolbeault cohomology groups $H^{p,q}(M)\coloneqq \ker(\bar\partial)/\Im(\bar\partial)$.
\item The sequence
\[
\Lambda^{p-1,q-1}(M)\xrightarrow[]{\partial\bar\partial}
\Lambda^{p,q}(M)\xrightarrow[]{(\partial,\bar\partial)}
\Lambda^{p+1,q}(M)\oplus\Lambda^{p,q+1}(M)
\]
defines the Bott-Chern cohomology spaces 
\[
H^{p,q}_{BC}(M)\coloneqq \left(\ker(\partial)\cap\ker(\bar\partial)\right)/\Im(\partial\bar\partial). 
\]
More simply: $H^{p,q}_{BC}(M)=\ker(d)/\Im(\partial\bar\partial)$.
\end{enumerate}
For a general complex manifold there exist only weak relationships between these cohomology groups and de Rham cohomology (see, for instance, \cite{AngellaTomassini, DGMS, Frolicher}).

\smallskip

Now assume $M$ is compact and K\"ahler. Then:
\begin{enumerate}[$\bullet$] 
\item The K\"ahler identities show that the Hodge Laplacian operator $\Delta \coloneqq dd^*+d^*d$ preserves types, \textit{i.e.}, it restricts to an operator $\Delta:\Lambda^{p,q}(M)\rightarrow\Lambda^{p,q}(M)$. 
This leads to the following decomposition of the spaces of complex harmonic forms
\[
\mathcal{H}^r=\Ker(\Delta|_{\Lambda^r})=\bigoplus_{p+q=r}\Ker(\Delta|_{\Lambda^{p,q}})=\bigoplus_{p+q=r}\mathcal{H}^{p,q}. 
\]
The Hodge decomposition theorems 
$$\Lambda^r(M)=\Im(\Delta)\oplus\Ker(\Delta), \quad \Lambda^{p,q}(M)=\Im(\Delta)\oplus\Ker(\Delta)$$
then show that $H^{p,q}(M)\simeq\mathcal{H}^{p,q}$ and that $H^r(M)\simeq\bigoplus_{p+q=r} H^{p,q}(M)$.
\item Bott-Chern cohomology is isomorphic to Dolbeault cohomology: 
$$H^{p,q}_{BC}(M)\simeq H^{p,q}(M).$$ 
This fact is, in some sense, the guiding light for this paper. It is a consequence of the global $\partial\bar\partial$ lemma. 
We shall investigate this in detail in the following sections.
\end{enumerate}

\begin{remark}
In particular, the above discussion shows the existence of an injection $H^{p,q}(M)\rightarrow H^r(M;\C)$, but its definition is rather involuted: 
it is obtained via a preliminary identification with $\mathcal{H}^{p,q}$. 
If however $p=q$ and we focus on real forms, then $\bar\partial\alpha=0$ implies $d\alpha=0$, 
so the map $[\alpha]\in H^{p,p}(M;\R)\mapsto [\alpha]\in H^r(M)$ is well-defined. 
It will follow from the global $\partial\bar\partial$ lemma, discussed below, that it is an injection.
\end{remark}

\begin{remark}
Notice, as also mentioned in the Introduction, that the K\"ahler condition is used to prove results concerning complex cohomologies. 
The $\U(n)$-decomposition of the spaces $\Lambda^r(M)$ leads to a further decomposition of the spaces of harmonic forms, 
such as (in the real case) $\mathcal{H}^{1,1}_{\R}=\R\cdot\omega\oplus\mathcal{H}^{1,1}_0$, 
but there does not exist a correspondingly refined differential complex, thus cohomology.
\end{remark}

\subsection{The local $\partial\bar\partial$ lemma.}

Recall the standard Poincar\'e lemma: on any smooth manifold, if $\alpha\in\Lambda^k(M)$ is $d$-closed then it is locally $d$-exact. 
The proof uses simple differential calculus.

Also recall the $\bar\partial$-Poincar\'e (or Dolbeault-Grothendieck) lemma: on any complex manifold, if $\alpha\in\Lambda^{p,q}(M)$ 
is $\bar\partial$-closed, then it is locally $\bar\partial$-exact. 
The proof is based upon integral representation formulae with singular holomorphic kernels, thus upon holomorphic function theory.
By conjugation, an analogous lemma holds for the $\partial$ operator.

These results have the following corollary.

\begin{prop}[Local $\partial\bar\partial$ lemma]\label{p:locallemma}
Let $(M,J)$ be a complex manifold. Fix $\alpha\in\Lambda^{p,q}(M)$. If $\alpha$ is $d$-closed then it is locally $\partial\bar\partial$-exact.
\end{prop}

Let us review the proof in the simplest setting: $\alpha\in\Lambda^{1,1}(M)$, $d\alpha=0$.

The standard Poincar\'e lemma implies that, locally, $\alpha=d\beta$. 
Let us decompose $\beta$ into types, $\beta=\beta^{1,0}+\beta^{0,1}$, and rewrite the previous identity as
\[
\alpha=\partial\beta^{1,0}+\partial\beta^{0,1}+\bar\partial\beta^{1,0}+\bar\partial\beta^{0,1}.
\]
Since the LHS is of type $(1,1)$, we find that
\[
\partial\beta^{1,0}=0, \quad \bar\partial\beta^{0,1}=0.
\]
The $\partial$-Poincar\'e lemma implies that $\beta^{1,0}=\partial f_1$. Analogously, $\beta^{0,1}=\bar\partial f_2$. 
It follows that $\alpha=\partial\bar\partial(f_2-f_1)$.
 
\begin{remark}
If the closed form $\alpha\in\Lambda^{1,1}(M)$ is also real, namely $\alpha=\overline{\alpha}$, then one further obtains 
$\partial\bar\partial(\mathrm{Re}(f_2-f_1))=0$, whence
\[
\alpha = i\partial\bar\partial(\mathrm{Im}(f_2-f_1)). 
\]
\end{remark}

It will be important for us later on to notice a second link between the local $\partial\bar\partial$ lemma and holomorphic functions, as follows.

The local $\partial\bar\partial$ lemma generates local potentials whose domains provide a covering of $M$. Two local potentials differ, on the intersection of their domains, by a pluriharmonic function:
\[
\alpha=\partial\bar\partial f_1=\partial\bar\partial f_2\quad\Rightarrow\quad\partial\bar\partial(f_1-f_2)=0.
\]
Proposition \ref{p:locallemma} thus implies the local existence of such functions.
We can verify their existence directly. For example, assume $f=re^{i\theta}$ is holomorphic and nowhere-vanishing (on some domain). 
Then $\log f=\log r+i\theta$ is holomorphic, so its real part $\log r=\log|f|$ is pluriharmonic.

\subsection{The global $\partial\bar\partial$ lemma.}

The local $\partial\bar\partial$ lemma holds in a very general setting: any complex manifold will do. However, its reach is very limited: as the name implies, it is a purely local result.

Certain categories of complex manifolds have an additional, surprisingly strong, property. 
The most relevant example is the class of compact, complex manifolds of K\"ahler type, \textit{i.e.}, which admit a K\"ahler metric. 
The result is as follows.

\begin{prop}[Global $\partial\bar\partial$ lemma]
Let $(M,J)$ be a compact complex manifold of K\"ahler type. Fix a $d$-closed $\alpha\in\Lambda^{p,q}(M)$. Then:
\begin{itemize}
\item The following conditions are equivalent: (i) $\alpha$ is globally $d$-exact, (ii) $\alpha$ is globally $\partial$-exact, (iii) $\alpha$ is globally $\bar\partial$-exact.
\item If any of the above conditions holds, then actually $\alpha$ is globally $\partial\bar\partial$-exact. 
 \end{itemize}
\end{prop}

The proof starts with the choice of a K\"ahler structure. It then relies on K\"ahler Hodge theory, which in turn relies on the Lefschetz decomposition (see, e.g., \cite[Cor.~3.2.10]{Huybrechts}). In particular, this proof is independent of the local 
$\partial\bar\partial$ lemma.

In this setting it is often useful to use an alternative expression for the $\partial\bar\partial$ operator. Specifically,
\[
2i\partial\bar\partial =   d\left(J^{-1}\right)^* d J^*,
\]
where we use the convention $J^*\alpha \coloneqq \alpha(J\cdot,\dots,J\cdot)$, for $\alpha\in\Lambda^k(M)$. 

The RHS emphasizes the possibility of moving away from the techniques of holomorphic function theory, which played an important role in the local lemma but is not relevant to the proof of the global lemma. It is this second formulation of the $\partial\bar\partial$ operator which, in the following sections, will become central to our extension of the $\partial\bar\partial$ lemma to other geometries. The point is that those other geometries do not offer an analogue of holomorphic function theory; in the $G_2$ case we will actually provide a counterexample to the local lemma.

\section{The Harvey-Lawson $\phi$-Hessian operator}\label{s:Hessian}

The goal of this section is to review the definition \cite{HL:intropotential} of the main operators of interest to us.

Let us consider a Riemannian manifold $(M,g)$ with a calibration $\phi\in\Lambda^k(M)$: by definition, this means $d\phi=0$ and 
$\phi|_\pi \leq \mathrm{vol}_\pi$, for all oriented tangent $k$-planes $\pi$ on $M.$

\begin{definition}[\cite{HL:intropotential}]
Let $(M,g)$ be a Riemannian manifold and let $\phi\in\Lambda^k(M)$ be a calibration. The $d^\phi$-operator is defined as
\[
d^\phi f \coloneqq \nabla f \lrcorner \phi, 
\]
for all $f\in \Lambda^0(M)$. The $dd^\phi$-operator is defined as $dd^\phi f\coloneqq d(d^\phi(f))\in\Lambda^k(M)$.
\end{definition}

The $dd^\phi$-operator is the highest-order term of a geometrically defined, second-order, differential operator, defined by Harvey and Lawson as follows. 

Consider the pointwise action of $\mathrm{Aut}(TM)$ on $\Lambda^k(M)$, 
\[
F^*\alpha \coloneqq \alpha(F\cdot,\dots,F\cdot).
\]
Let us focus on the orbit (more precisely: the collection of pointwise orbits)  
of $\alpha=\phi$, so that the isotropy subgroup $\mathrm{Aut}_\phi\subseteq \mathrm{Aut}(TM)$
is the group of automorphisms which preserve $\phi$ at each point. 
The linearized action is 
\[
\lambda_\phi:\mathrm{End}(TM)\to\Lambda^k(M),\quad 
\lambda_\phi(A) = A\cdot\phi = \phi(A\cdot,\ldots,\cdot) + \cdots + \phi(\cdot,\ldots,A\cdot).
\]
The kernel of this action is, at each point, the Lie algebra of the isotropy subgroup; its image is the tangent space of the orbit. 

Decomposing $A$ into its symmetric and skew-symmetric parts, $A=A^{\mathrm{sym}}+A^{\mathrm{skew}}$, 
provides a corresponding decomposition of $\lambda_\phi(A)$.

Let us focus on the following special case.
Given $Y\in\Gamma(TM)$, we shall set $A=\nabla Y$. A computation then shows
\[
\lambda_\phi(\nabla Y) = \mathcal{L}_Y\phi -\nabla_Y\phi = d(Y\lrcorner \phi) -\nabla_Y\phi. 
\]
Indeed, 
\[
\begin{split}
\lambda_\phi(\nabla Y)(X_1,\ldots,X_k) &= \phi(\nabla_{X_1}Y,X_2,\ldots,X_{k}) + \cdots +   \phi(X_1,\ldots,\nabla_{X_{k}}Y)\\
								 &= \phi(\nabla_{Y}X_1,X_2,\ldots,X_{k}) + \phi([X_1,Y],X_2,\ldots,X_{k})\\
								 &\quad +\cdots+ \phi(X_1,\ldots,\nabla_Y{X_{k}}) +  \phi(X_1,\ldots,[{X_{k}},Y]) \\
								 &= \phi(\nabla_{Y}X_1,X_2,\ldots,X_{k}) + \cdots + \phi(X_1,\ldots,\nabla_Y{X_{k}}) \\
								 &\quad -\phi([Y,X_1],X_2,\ldots,X_{k})-\cdots- \phi(X_1,\ldots,[Y,{X_{k}}])\\
								 &= (-\nabla_Y\phi+\mathcal{L}_Y\phi)(X_1,\ldots,X_k), 
\end{split}
\]
where the last equality follows from adding and subtracting $Y(\phi(X_1,\ldots,X_k))$. 

Recall that
\[
(\nabla Y)^{\mathrm{sym}} = \frac12 \mathcal{L}_Yg,\qquad (\nabla Y)^{\mathrm{skew}} = \frac12 dY^\flat, 
\]
where we are identifying symmetric and skew-symmetric endomorphisms with symmetric and skew-symmetric 2-covariant tensors, respectively. 

We are especially interested in the case where $Y=\nabla f = (df)^\sharp$ is the gradient of a smooth function $f:M\rightarrow\R$. Then 
\[
\begin{split}
(\nabla Y)^{\mathrm{sym}}		&= \frac12 \mathcal{L}_{\nabla f}g = \mathrm{Hess}_g(f),\\
(\nabla Y)^{\mathrm{skew}} 	&= \frac12 d(\nabla f)^\flat = \frac12ddf =0.
\end{split}
\]
This explains why, from now on, we shall focus on symmetric endomorphisms. It also leads to Harvey-Lawson's formula 
\[
\lambda_\phi(\nabla\nabla f) = d(\nabla f\lrcorner \phi) -\nabla_{\nabla f}\phi.
\] 

\begin{definition}[\cite{HL:intropotential}]\label{Def:phiHess}
Let $(M,g)$ be a Riemannian manifold and let $\phi\in\Lambda^k(M)$ be a calibration. The $\phi$-Hessian of $f\in \Lambda^0(M)$ is given by
\[
\mathcal{H}^\phi(f) = \lambda_\phi(\nabla\nabla f) = dd^\phi f -\nabla_{\nabla f}\phi.
\]
In particular, if $\phi$ is parallel the $\phi$-Hessian reduces to the $dd^\phi$-operator.  
\end{definition}

\begin{remark}
Notice that the standard Hessian, $\mathrm{Hess}_g(f)$, is a symmetric form. The $\phi$-Hessian, $\mathcal{H}^\phi(f)$, is an alternating form. 

We will see below that, on K\"ahler manifolds, the $\phi$-Hessian is equivalent to 
the $i\partial\bar\partial$ operator. 
As seen above, this operator is usually not defined in terms of group actions. 
Harvey-Lawson's emphasis on the role of group actions thus adds an interesting twist even to the classical theory. 
\end{remark}

\section{K\"ahler manifolds, revisited}

Let $(M,J,\omega)$ be a compact K\"ahler manifold of real dimension $2n$, with K\"ahler metric $h = \omega(\cdot,J\cdot)$. 
The 2-form $\omega$ is a calibration so we are in the setting of Section \ref{s:Hessian} with $k=2$ and $\phi=\omega$. 
Our first goal is to review the $\lambda_\phi$ map, here denoted $\lambda_\omega$, and relate/adapt it to the specific circumstances of K\"ahler geometry. 
For the reason mentioned in Section \ref{s:Hessian}, we shall restrict it to symmetric endomorphisms.

Since $\nabla\omega=0$, the $\omega$-Hessian reduces to $dd^\omega$. Furthermore, 
\[
d^\omega(f)=\omega(\nabla f,\cdot)=-df\circ J(\cdot) =  (J^{-1})^* df
\]
so $dd^\omega= d(J^{-1})^*d=2i\partial\bar\partial$. In particular, $\Im(dd^\omega)\subseteq \Lambda^{1,1}_\R(M)$.
The Harvey-Lawson theory thus reduces to the classical, complex, theory of the $\partial\bar\partial$ operator: 
the dependence on $\omega$ is only superficial. 
However, as mentioned, in the compact K\"ahler case this operator has stronger properties than usual, exemplified by the global $\partial\bar\partial$ lemma. 
Our second goal is to provide a direct proof of this result, using the Harvey-Lawson formalism.

A final comment: the relationship between $dd^\omega f$ and $\Hess(f)$ is explicit in the definition of the $\omega$-Hessian. Alternatively, one can check that
\[
dd^\omega f(X,JX)=X(Xf)+JX(JX f)+df(J[X,JX]).
\]
Furthermore,
\[
J[X,JX] =J(\nabla_X JX-\nabla_{JX}X) = -\nabla_XX-\nabla_{JX}JX.
\]
These two equations lead to the conclusion that 
\[
dd^\omega f(X,JX)=\Hess(f)(X,X)+\Hess(f)(JX,JX).
\]

\subsection{Linear algebra}\label{SecKahlerLA}

Let $V$ be a real $2n$-dimensional vector space endowed with a complex structure $J$.
Let $S^2V^*$ denote the space of real symmetric 2-forms on $V$. 
Set
\[
\begin{split}
S^2_J		&\coloneqq\{h\in S^2V^*~:~J^*h=h\},\\
\Lambda_\R^{1,1}	&\coloneqq \{\alpha\in\Lambda^{1,0}\wedge\Lambda^{0,1}~:~\alpha=\bar\alpha\}
				=\{\alpha\in\Lambda^2V^*~:~J^*\alpha=\alpha\}.
\end{split}
\]
The group $\GL(V,J)\simeq\GL(n,\C)$ acts irreducibly on both spaces. 

\begin{lem} The map
\[
\mathrm{i}_J:S^2_J \to \Lambda_\R^{1,1},\quad h\mapsto h(J\cdot,\cdot)
\]
is a $\GL(V,J)$-equivariant isomorphism. 
Its inverse is the map $\alpha\mapsto\alpha(\cdot, J\cdot)$.
\end{lem}

For the purposes of K\"ahler geometry we shall also fix a positive $h\in S^2_J$, corresponding to some $\omega=\mathrm{i}_J(h)\in\Lambda_\R^{1,1}$. 
Consider the corresponding subgroup $\U(h)\cong\U(n)$. It preserves the line $\R \cdot h$, equivalently $\R \cdot \omega$, 
leading to splittings into U$(h)$-invariant irreducible subspaces 
\[
S^2_J=\R\cdot h\oplus S^2_{J,0}, \qquad \Lambda_\R^{1,1}=\R\cdot\omega\oplus \Lambda^{1,1}_0,
\]
where $S^2_{J,0} = \{a\in S^2_J ~:~ \mathrm{tr}_ha=0 \}$.

Now consider the space $\mathrm{Sym}(V)$ of symmetric endomorphisms. It decomposes into irreducible U$(h)$-invariant subspaces as follows:
\[
\mathrm{Sym}(V) = \R\cdot \mathrm{Id} \oplus \mathrm{Sym}^{+}_{0} \oplus  \mathrm{Sym}^{-}, 
\]
where 
\[
\begin{split}
\mathrm{Sym}^{+}_{0}	&\coloneqq \left\{A \in \mathrm{Sym}(V) ~:~ AJ = JA \mbox{ and } \tr(A)=0 \right\},\\
\mathrm{Sym}^{-}		&\coloneqq \left\{A \in \mathrm{Sym}(V) ~:~  AJ = -JA \right\}.
\end{split}
\]
Note that $\R\cdot \mathrm{Id} \oplus \mathrm{Sym}^{+}_{0}=\mathrm{Sym}_J$ is the space of symmetric endomorphisms commuting with $J$. 
It is clearly isomorphic to $S^2_J=\R\cdot h\oplus S^2_{J,0}$ via 
\[
s:\mathrm{Sym}_J\to S^2_J,\quad s(A) = h(A\cdot,\cdot).
\] 

\smallskip

The restriction of the map $\lambda_\omega:\mathrm{End}(V)\to\Lambda^2V^*$ to $\mathrm{Sym}(V)$ is given by
\[
\lambda_\omega(A)	= \omega(A\cdot,\cdot) + \omega(\cdot,A\cdot) = h(JA\cdot,\cdot) + h(J\cdot,A\cdot) = h((JA+AJ)\cdot,\cdot), 
\]
for every $A\in\mathrm{Sym}(V)$.
From this, we see that its kernel coincides with $\mathrm{Sym}^{-}$ and we obtain the following. 
\begin{lem}
The map $\mathrm{i}_\omega\coloneqq \frac12 \lambda_\omega|_{\mathrm{Sym}_J}$ gives an isomorphism
\[
\mathrm{i}_\omega : \R\cdot \mathrm{Id} \oplus \mathrm{Sym}^{+}_{0} \to \R\cdot\omega \oplus \Lambda^{1,1}_0,\quad
\mathrm{i}_\omega(A) = h(JA\cdot,\cdot) = \omega(A\cdot,\cdot). 
\]
In particular, $\mathrm{i}_J=\mathrm{i}_\omega\circ s^{-1}$. 
\end{lem} 

\begin{example}
 
The $S^2_J$-component of $\Hess(f)$ is $1/2(\Hess(f)+J^*\Hess(f))$. 
Applying $\mathrm{i}_J$ confirms the above relationship between the $\omega$-Hessian and the standard Hessian.
Roughly speaking: if we take into account the projection and the fact that $\mathrm{i}_J$ is an isomorphism, 
we may say that the $\omega$-Hessian is equivalent to ``half'' of the standard Hessian.
\end{example}

\subsection{The global $\partial\overline{\partial}$ lemma, revisited}\label{SecKahlerdd}

Recall from Section \ref{s:dedebar} the statement of the global $\partial\bar\partial$ lemma. 
We are mainly interested in the case of $(1,1)$ forms. Roughly speaking, in this case the lemma can be rephrased as follows: 
the images of the operators $d,\partial,\bar\partial$ and $\partial\bar\partial$ coincide, after intersection with $\Lambda^{1,1}(M)\cap\Ker(d)$.
The proof further shows that this space can be characterized as the orthogonal space $(\mathcal{H}^{1,1})^\perp$, again
restricted to $\Lambda^{1,1}(M)\cap\Ker(d)$.

\begin{remark}
The above operators certainly do not have the same image in the full space $\Lambda^{2}(M,\C)$. The result depends on the very strong constraints enforced by the two restrictions: type $(1,1)$, closed.
\end{remark}

This formulation brings us closer to the Hodge decomposition viewpoint of this paper. The goal of this section is to complete this transition, also reversing the process by showing how to first prove the Hodge decomposition, then obtain the $\partial\bar\partial$ lemma as a corollary. 

The formulation of the $\partial\bar\partial$ lemma given here allows us to restrict to real $(1,1)$ forms and functions, and is optimally suited to extensions to other geometries. For this reason, we shall provide full details of the proof. We will use the fact that, on K\"ahler manifolds, the Hodge Laplacian preserves types. 
This follows from the K\"ahler identities, specifically from the fact that the Hodge Laplacian coincides with the Dolbeault Laplacian, up to a constant. 
It is independent of K\"ahler Hodge theory.

\begin{lem}\label{l:Kinjectivity}
The principal symbol $\sigma(dd^\omega)$ of the linear differential operator $dd^\omega:\Lambda^0(M)\rightarrow\Lambda^{1,1}_\R(M)$ 
is given by the homomorphisms
\[ 
\sigma(dd^\omega)|_{x}(\xi):\R\to \Lambda^{1,1}_\R(T^*_xM),\quad \sigma(dd^\omega)|_{x}(\xi):c \mapsto c\,\xi\wedge(\xi^\sharp \lrcorner \omega),
\]
for all $x\in M$ and $\xi \in T^*_xM$. 
This symbol is injective. 
\end{lem}
\begin{proof} 
Let $x\in M$ and choose a unitary coframe $(e^1,\ldots,e^{2n})$ of $T^*_xM$ so that 
\[
\omega|_x = \sum_{k=1}^n e^{2k-1}\wedge e^{2k}.
\]
Let $\xi\in T^*_xM\smallsetminus\{0\}$. 
Since $\mathrm{U}(n)$ acts transitively on the unit sphere in $T^*_xM$, we may assume that $\xi = a e^1$, with $a\neq 0$. Then 
\[
\xi\wedge(\xi^\sharp \lrcorner \omega) = a^2\,e^{12}.  
\]
The statement then follows. 
\end{proof}

\begin{thm}\label{ThmSplittingKahler}
Let $(M,J,\omega)$ be a compact K\"ahler manifold.
Then there exists an $L^2$-orthogonal decomposition 
\[
\Lambda^{1,1}_\R(M)\cap \Ker(d)=\Im(dd^\omega)\oplus\mathcal{H}^{1,1}_\R, 
\]
where $\mathcal{H}^{1,1}_\R$ denotes the space of real harmonic $(1,1)$ forms.
\end{thm}

\begin{proof}
We shall split the proof into several steps.

\begin{enumerate}[1.]

\item Consider the operator
\[
dd^\omega:\Lambda^0(M)\rightarrow \Lambda^{1,1}_\R(M).
\]
As shown in Appendix \ref{Sect:AnalyticTools}, we can act on it in two different ways. Let us first extend it to an operator
\[
dd^\omega:H^2(M)\rightarrow L^2(\Lambda^{1,1}_\R(M)).
\]
Then, Lemma \ref{l:Kinjectivity} implies that it has closed image.

Now let us view it as an unbounded operator $dd^\omega:L^2(M)\rightarrow L^2(\Lambda^{1,1}_\R(M))$, using the domain $D = H^2(M)$. 
We can then apply Proposition \ref{p:abstractdecomp} to obtain a decomposition
\[
L^2(\Lambda^{1,1}_\R(M))=\Im(dd^\omega)\oplus\Ker((dd^\omega)^*),
\]
where $((dd^\omega)^*,D')$ is the adjoint of $(dd^\omega,H^2(M))$ and its domain $D'$ satisfies $\Lambda^{1,1}_\R(M)\subseteq D'\subseteq L^2(\Lambda^{1,1}_\R(M))$.

\item Recall that $(dd^\omega)^*$ is a distributional extension of $(dd^\omega)^t$. Our next task is to provide an explicit expression for $(dd^\omega)^t$.

Let $\alpha$ denote the generic element in $\Lambda^{1,1}_\R(M)$. Given any $f\in\Lambda^0(M)$,
\begin{align*}
\int_M (dd^\omega f,\alpha)\vol_h&=\int_M (\nabla f\lrcorner\omega,d^*\alpha)\vol_h\\
&=\int_M (\star(\nabla f\lrcorner\omega),\star d^*\alpha)\vol_h\\
&=-\frac{1}{(n-1)!} \int_M (df\w\omega^{n-1},\star d^*\alpha)\vol_h\\
&=\frac{1}{(n-1)!}\int_Mdf\wedge\omega^{n-1}\wedge d^*\alpha\\
&=-\frac{1}{(n-1)!}\int_Mf\,  d(\omega^{n-1}\wedge d^*\alpha)\\
&=-\frac{1}{(n-1)!}\int_Mf\, dd^*\alpha\wedge\omega^{n-1}\\
&=-\frac{1}{(n-1)!}\int_M(f,\star(dd^*\alpha\wedge\omega^{n-1}))\vol_h.
\end{align*}
This proves that $(dd^\omega)^t: \Lambda^{1,1}_\R(M)\rightarrow\Lambda^0(M)$ has the following expression
\[
(dd^\omega)^t = {-\frac{1}{(n-1)!}}\star(dd^*(\cdot)\wedge\omega^{n-1}) 
= -\star(dd^*(\cdot)\wedge\star\omega).
\]

\item We now want to refine our understanding of the above operators.

Consider the operator $d:\Lambda^2(M)\rightarrow\Lambda^3(M)$. 
Recall that the operator $dd^\omega$ actually takes values in the subspace $\Ker(d)$, so we can reformulate it as a map
\[
dd^\omega:\Lambda^0(M)\rightarrow\Lambda^{1,1}_\R(M)\cap\Ker(d).
\]
Let $\overline{\Ker(d)}$ denote the $L^2$-closure of $\Ker(d)$. 
The {RHS} is contained in the complete space $L^2(\Lambda^{1,1}_\R(M))\cap\overline{\Ker(d)}$. 
This means that, extending $dd^\omega$ by continuity, we can be more precise than we were previously: we obtain a map
\[
dd^\omega:H^2(M)\rightarrow L^2(\Lambda^{1,1}_\R(M))\cap\overline{\Ker(d)}.
\]
Alternatively, we can think of it as an unbounded map between Hilbert spaces
\[
dd^\omega:L^2(M)\rightarrow L^2(\Lambda^{1,1}_\R(M))\cap\overline{\Ker(d)}
\]
with domain $H^2(M)$. Its adjoint is simply the restriction of the adjoint $(dd^\omega)^*$ found previously:
\[
(dd^\omega)^*|_{\overline{\Ker(d)}}: {L^2(\Lambda^{1,1}_\R(M))\cap\overline{\Ker(d)}} \rightarrow L^2(M),
\]
with domain $D' \cap \overline{\Ker(d)}$. 
In this new setting Proposition \ref{p:abstractdecomp} provides the decomposition
\[
L^2(\Lambda^{1,1}_\R(M))\cap\overline{\Ker(d)}=\Im(dd^\omega)\oplus \Ker\left((dd^\omega)^*|_{\overline{\Ker(d)}}\right).
\]
\item It is convenient to incorporate the restriction to $\overline{\Ker(d)}$ directly into the dual operator. We shall do this in this step and the next.

Consider the auxiliary operator
\[
Q:\Lambda^{1,1}_\R(M)\rightarrow\Lambda^0(M)\oplus\Lambda^3(M), \quad Q: \alpha\mapsto((dd^\omega)^t(\alpha),d\alpha).
\]

Notice that $\Ker(Q)=\Ker ((dd^\omega)^t,d) =  \Ker\left((dd^\omega)^t|_{\Ker(d)}\right)$. 

The principal symbol $\sigma(Q)$ is given by the homomorphisms
\[ 
\sigma(Q)|_x(\xi) : \alpha \mapsto \left(-\star(\xi\wedge(\xi^\sharp\lrcorner\alpha)\wedge\star\omega),\xi\wedge\alpha\right), 
\]
for each $x\in M$ and $\xi \in T^*_xM.$ 
Assume that $\xi\neq 0$ and $\alpha\in \ker(\sigma(Q)|_x(\xi))$, then 
\[
\xi \wedge(\xi^\sharp\lrcorner\alpha) = |\xi|^2 \alpha - \xi^\sharp\lrcorner (\xi \wedge \alpha) = |\xi|^2 \alpha, 
\]
and thus
\[
0 = \star(\xi\wedge(\xi^\sharp\lrcorner\alpha)\wedge\star\omega) = |\xi|^2 \star(\alpha\wedge \star\omega).
\]
Therefore $\alpha\in \Lambda^{1,1}_0(T^*_xM)$. 
Choose a unitary coframe $(e^1,\ldots,e^{2n})$ of $T^*_xM$ and use the $\mathrm{U}(n)$ action to reduce to the case $\xi = a e^1$, 
with $a\neq0$. It is then simple to show that, for every $\alpha\in\Lambda^{1,1}_\R(T^*_xM)$, 
the condition $\xi\wedge\alpha = 0$  gives $\alpha = c\,e^{12}$, for some $c\in\R$.   
Comparing this with the condition $\alpha\in \Lambda^{1,1}_0(T^*_xM)$ shows that $c=0$. 
Therefore the principal symbol is injective, so the usual regularity property holds. 

In particular, consider the distributional extension of $Q = ((dd^\omega)^t,d)$ defined by 
\[
((dd^\omega)^*,d) : D' \to L^2(\Lambda^0(M))\oplus\Lambda^3(M)',
\]
where $D'$ is defined as in Step 1 and we extend $d$ as in the Appendix, cf.~\eqref{eq:distr_extension}, then restrict it to $D'$.
Then $\Ker((dd^\omega)^*,d) = \Ker (Q)$, so it consists of smooth forms. 

\item We now want to show that $\Ker\left((dd^\omega)^*|_{\overline{\Ker(d)}}\right)=\Ker(Q)$.

We claim that
\[
\Ker\left((dd^\omega)^*|_{\overline{\Ker(d)}}\right) \subseteq \Ker((dd^\omega)^*,d)=\Ker(Q). 
\]
Indeed, any $\alpha$ in the LHS belongs to $D'\cap\overline{\Ker(d)}$, so there exists a sequence 
$\{\alpha_n\}\subset \Lambda^{1,1}_\R(M)$ converging to $\alpha$ such that $d(\alpha_n)=0$. 
Then, for any test form $\beta\in\Lambda^3(M)$ we have
\[
d\alpha(\beta) = \alpha(d^*\beta) = \int_M (\alpha, d^*\beta) = \lim_{n\to\infty}\int_M (\alpha_n, d^*\beta) = \lim_{n\to\infty}\int_M (d\alpha_n, \beta) = 0. 
\] 
On the other hand, it is clear that 
\[
\Ker\left((dd^\omega)^*|_{\overline{\Ker(d)}}\right) \supseteq \Ker\left((dd^\omega)^t|_{\Ker(d)}\right)=\Ker(Q).
\]
This proves the desired equality.
\item We shall now characterize $\Ker(Q)$. 
As a first step, notice that 
\[
(dd^\omega)^t|_{\Ker(d)}=-\frac{1}{(n-1)!}\star(\Delta|_{\Ker(d)}(\cdot)\wedge\omega^{n-1}),
\]
where $\Delta=dd^*+d^*d : \Lambda^{1,1}_\R(M)\to \Lambda^{1,1}_\R(M)$, since $\Delta$ preserves types. 
It follows that
\[
\Ker\left((dd^\omega)^t|_{\Ker(d)}\right)=\Ker(\star(\Delta(\cdot)\wedge\omega^{n-1}))\cap\Ker(d) 
=\Ker(\Delta(\cdot)\wedge\omega^{n-1})\cap\Ker(d).
\]
Recall the splitting 
\[
\Lambda^{1,1}_\R(M)=\Lambda^0(M)\cdot\omega\oplus\Lambda^{1,1}_0(M).
\]
Since $\Delta$ preserves types and the Lefschetz decomposition, and $\Lambda^{1,1}_0(M)=\Ker((\cdot)\wedge\omega^{n-1}))$, we see that 
\[
\Ker(\Delta(\cdot)\wedge\omega^{n-1}) = \R\cdot\omega\oplus\Lambda^{1,1}_0(M), 
\]
whence
\[
\Ker(\Delta(\cdot)\wedge\omega^{n-1})\cap \Ker(d)  
=\R\cdot\omega\oplus\left(\Lambda^{1,1}_0(M)\cap \Ker(d)\right). 
\]
Now, every $\alpha\in\Lambda^{1,1}_0(M)$ satisfies the identity
\[
\star\alpha=\frac{-1}{(n-2)!}\omega^{n-2}\wedge\alpha
\]
(this linear algebra statement holds in any Hermitian vector space). 
In particular, if $\alpha$ is closed then it is co-closed, thus harmonic. This implies that 
\[
\Ker(\Delta(\cdot)\wedge\omega^{n-1})\cap \Ker(d) = \R\cdot\omega\oplus\left(\Lambda^{1,1}_0(M)\cap \Ker(d)\right) = \R\cdot\omega\oplus \mathcal{H}^{1,1}_0 = \mathcal{H}^{1,1}_\R. 
\]

\item Using the previous steps, we can now update the decomposition given at the end of Step 3:
\[
L^2(\Lambda^{1,1}_\R(M))\cap\overline{\Ker(d)}=\Im(dd^\omega)\oplus \mathcal{H}^{1,1}_\R.
\]
This decomposition is formulated in terms of the extended operator $dd^\omega:L^2(M)\rightarrow L^2(\Lambda^{1,1}_\R(M))\cap\overline{\Ker(d)}$. 
Our final step is to obtain a decomposition concerning the original operator 
$dd^\omega:\Lambda^0(M)\rightarrow\Lambda^{1,1}_\R(M)\cap\ker(d)$. 
The idea is to intersect both sides of the new decomposition with $\Lambda^{1,1}_\R(M)$. 

Notice that $\mathcal{H}^{1,1}_\R\subseteq\Lambda^{1,1}_\R(M)$. It follows that
\begin{align*}
\Lambda^{1,1}_\R(M)\cap\ker(d)&=\left(L^2(\Lambda^{1,1}_\R(M))\cap\overline{\Ker(d)}\right)\cap\Lambda^{1,1}_\R(M)\\
&=\left(\Im(dd^\omega)\oplus \mathcal{H}^{1,1}_\R\right)\cap\Lambda^{1,1}_\R(M)\\
&=\left(\Im(dd^\omega)\cap\Lambda^{1,1}_\R(M)\right)\oplus\mathcal{H}^{1,1}_\R.
\end{align*}
Here we are still using the extended operator $dd^\omega$ defined on $L^2(M)$. However, regularity implies that 
\[
\Im(dd^\omega)\cap\Lambda^{1,1}_\R(M)=\Im(dd^\omega),
\]
where now on the RHS we are using the operator $dd^\omega$ defined on $\Lambda^0(M)$.
\end{enumerate}
\end{proof}

Theorem \ref{ThmSplittingKahler} implies the following.

\begin{cor}[Global $dd^\omega$ lemma]\label{cor:ddomega}
Let $(M,J,\omega)$ be a compact K\"ahler manifold. 
Fix a real form $\alpha\in\Lambda^{1,1}_\R(M)$. If $\alpha$ is globally $d$-exact, then it is globally $dd^\omega$-exact: 
$\alpha=dd^\omega f$, for some $f\in \Lambda^0(M)$. 
\end{cor}
\begin{proof}
Recall that on a compact Riemannian manifold exact forms are orthogonal to harmonic forms. 
Recall also that $\mathcal{H}^2(M)\cap \Lambda^{1,1}_\R(M) = \mathcal{H}^{1,1}_\R$. It follows that 
$\Im(d)\cap\Lambda^{1,1}_\R(M)$ is orthogonal to $\mathcal{H}^{1,1}_\R$, so it is contained in $\Im(dd^\omega)$. 
The opposite inclusion is obvious.
\end{proof}

\begin{remark}\label{rem:alternativeKahler}
Let us outline an alternative proof of Theorem \ref{ThmSplittingKahler} which relies instead on the notion of elliptic complexes \cite{AtBo}.
 For every $x\in M$ and $\xi\in T^*_xM\smallsetminus\{0\}$ we have
\[
\Im(\sigma(dd^\omega)|_x(\xi)) 	= \R\cdot \xi \wedge (\xi^\sharp\lrcorner \omega)  
						= \ker(\sigma(d|_{\Lambda^{1,1}_\R(M)})|_x(\xi)).
\]
The first identity follows immediately from Lemma \ref{l:Kinjectivity}. 
The second identity can be shown via the same argument used above to prove the injectivity of the principal symbol of the auxiliary operator $Q$. 
This shows that the following complex is elliptic at its middle term, in the sense that its symbol sequence is exact there:
\[
\Lambda^0(M)\xrightarrow{dd^\omega}{}\Lambda^{1,1}_{\R}(M)\xrightarrow{d}{}\Lambda^3(M). 
\]
The theory of elliptic complexes then implies the existence of a fourth order linear differential operator which is both elliptic and self-adjoint:
\[
\Box \coloneqq (dd^\omega)(dd^\omega)^* + d^*dd^*d : \Lambda^{1,1}_\R(M) \to \Lambda^{1,1}_\R(M).  
\]
This yields the $L^2$-orthogonal splitting 
\[
\Lambda^{1,1}_\R(M)  = \Im(\Box)\oplus\ker(\Box) =  \mathrm{im}(dd^\omega)\oplus \mathrm{im}(d^*) \oplus (\mathrm{ker}(d) \cap \mathrm{ker}((dd^\omega)^*)). 
\] 
The proof of Theorem \ref{ThmSplittingKahler} shows that $\mathrm{ker}(d) \cap \mathrm{ker}((dd^\omega)^*) = \mathcal{H}^{1,1}_\R$. 
The statement of the theorem then follows after intersecting both sides of this decomposition with $\ker(d)$. 
\end{remark}

\subsection{Cohomology.}\label{sec:cohomologyBC}

Let $M$ be a K\"ahler manifold. The sequence 
\[
\Lambda^0(M)\xrightarrow{dd^\omega}{}\Lambda^{1,1}_{\R}(M)\xrightarrow{d}{}\Lambda^3(M)
\]
produces the cohomology space $H^{1,1}_{BC}(M;\R)\coloneqq \Ker(d)/\Im(dd^\omega)$. 
Its complexification coincides with the usual Bott-Chern cohomology space $H^{1,1}_{BC}(M)$. We can rephrase the above as follows.

\begin{prop}\label{prop:BC}
Assume $M$ is compact K\"ahler. Then the linear map
\[
H^{1,1}_{BC}(M;\R)\rightarrow H^2(M;\R), \quad [\alpha]_{BC}\mapsto [\alpha]
\]
is an injection which identifies the real Bott-Chern cohomology with the subspace of de Rham classes represented 
by the real harmonic forms $\mathcal{H}^{1,1}_\R$.
\end{prop}

This provides another cohomological characterization of $\mathcal{H}^{1,1} = \mathcal{H}^{1,1}_\R \otimes \mathbb{C}$, 
alternative to that furnished by complex Hodge theory in terms of Dolbeault cohomology.

\section{Torsion-free G$_\mathbf{2}$ 7-manifolds}
\subsection{Linear algebra} 
Let $(V,\ph)$ denote a $G_2$ vector space, endowed with the induced orientation and metric $g_\ph$, and let $\psi \coloneqq \star\ph$. 
Let $(e^1,\ldots,e^7)$ be a $G_2$ coframe, namely an orthonormal basis of $V^*$ such that
\[
\ph	=	e^{123} + e^{145} + e^{167} + e^{246} - e^{257} - e^{347} - e^{356}
		= \frac16 \ep_{ijk}e^{ijk}, 
\]
where $\ep_{ijk}$ is the unique symbol that is totally skew-symmetric in its indices and satisfies $\ep_{123}=1$ etc.  

As in the complex case, there exists an interesting link between $S^2V^*$ and a certain space of $k$-forms. 

Recall the decomposition of $\Lambda^3V^*$ into $G_2$-invariant and irreducible subspaces
\[
\Lambda^3V^* = \Lambda^3_1 \oplus \Lambda^3_7 \oplus \Lambda^3_{27},
\]
where
\[
\begin{split}
\Lambda^3_1 &\coloneqq \R\,\ph,\quad
\Lambda^3_{7}\coloneqq \{\star(\alpha\w\ph)~:~\alpha\in\Lambda^1V^*\}, \\
\Lambda^3_{27}&\coloneqq \{\gamma\in\Lambda^3V^*~:~\gamma\wedge\ph=0,~\gamma\wedge\psi=0\}.
\end{split}
\]
Notice that
\[
\Lambda^3_1 \oplus \Lambda^3_{27} = \{\gamma\in\Lambda^3V^*~:~\gamma\wedge\ph=0\},
\]
since $\star(\alpha\w\ph)\w\ph = -4\star\alpha$, for every $\alpha\in\Lambda^1V^*$. 

\begin{lem}[\cite{Bryant}]
The maps $\i_\ph:S^2V^*\rightarrow\Lambda^3_1\oplus\Lambda^3_{27}$ 
\[
\begin{split}
\i_\ph(h_{ij}e^ie^j)	&= \sum_{j,k,l =1}^7\sum_{r=1}^7 h_{r[j|}\ep_{r|kl]}e^{jkl} \\
									& = 2\sum_{j<k<l} \sum_{r=1}^{7} \left(\ep_{rkl}h_{rj} + \ep_{rjk}h_{rl}+\ep_{rlj}h_{rk}\right) e^{jkl},
\end{split}
\]
and $\j_\ph:\Lambda^3_1\oplus\Lambda^3_{27}\rightarrow S^2V^*$
\[
\j_\ph(\gamma)(v,w) =  \star( v\lrcorner \ph \W w\lrcorner \ph \W \gamma), 
\]
are $G_2$-equivariant isomorphisms.
\end{lem}

For example, $\i_\ph(g_\ph)=6\ph$ and $\j_\ph(\ph) = 6 g_\ph$. Note: these isomorphisms are not inverses of each other. One has 
\[
\i_\ph(\j_\ph(\gamma)) = 8\gamma + 4\star(\ph\wedge \star\gamma)\,\ph, 
\]
and 
\[
\j_\ph(\i_\ph(h)) = 8h +  4\tr_{g_\ph}(h)\, g_\ph. 
\]
The space of symmetric endomorphisms decomposes into $G_2$-invariant and irreducible subspaces as
\[
\mathrm{Sym}(V) = \R\,\mathrm{Id} \oplus \mathrm{Sym}_0.  
\]
As in the K\"ahler case, let $s:\mathrm{Sym}(V)=\R\,\mathrm{Id} \oplus \mathrm{Sym}_0\to S^2V^* = \R g_\ph \oplus S^2_0$ denote 
the ($G_2$-equivariant) isomorphism $s(A) = g(A\cdot,\cdot)$. We then have the following. 

\begin{lem}\label{prop:imap}
The map $\i_\ph$ coincides with the map $2 \lambda_\ph|_{\mathrm{Sym}(V)} \circ s^{-1}$. 
\end{lem}
\begin{proof}
Let $h\in S^2V^*$ and let $H\coloneqq s^{-1}(h)$. 
Given any triple of orthonormal vectors $e_j$, $e_k$, $e_l$ with $1\leq j<k<l\leq7$, we have
\[
\begin{split}
\i_\ph(h)(e_j,e_k,e_l) 	&= 2 \sum_{r=1}^{7} \left(\ep_{rkl}h_{rj} + \ep_{rjk}h_{rl}+\ep_{rlj}h_{rk}\right)  \\		
								&= 2\left(\ph(H(e_j),e_k,e_l) + \ph(H(e_l),e_j,e_k)+\ph(H(e_k),e_l,e_j)\right)\\
								&= 2 (H \cdot \ph)(e_j,e_k,e_l) = 2\lambda_{\ph}(H)(e_j,e_k,e_l).
\end{split}
\] \end{proof}

\begin{remark}
The above provides a different point of view on the splitting $\Lambda^3V^*=(\Lambda^3_1\oplus\Lambda^3_{27})\oplus\Lambda^3_7$, 
alternative to the one provided by the action of the group $G_2$. 
\end{remark}

\subsection{The decomposition of $d$}
Analogously to the complex/K\"ahler case, the spaces of differential forms on a 7-manifold $M$ with a $G_2$-structure $\ph$ 
decompose into $G_2$-irreducible subspaces. 
The standard notation is as follows:
\[
\Lambda^2(M)= \Lambda^2_7(M) \oplus \Lambda^2_{14}(M),\qquad 
\Lambda^3(M)= \Lambda^3_1(M) \oplus \Lambda^3_7(M) \oplus \Lambda^3_{27}(M), 
\]
where the lower index refers to the rank of the corresponding subbundle. Notice that the splitting is orthogonal with respect to the metric $g_\ph$ 
induced by $\ph$. We shall denote the orthogonal projections by $\pi^k_q:\Lambda^k(M)\to \Lambda^k_q(M)$. 

In turn, these decompositions define decompositions of the exterior differential operator $d$. 
Following \cite{Bryant}, we restrict our attention to torsion-free $G_2$-structures, i.e., those satisfying $d\ph=0$ and $d\psi=0$, 
and we let
\begin{align*}
\Lambda_1	&\coloneqq \Lambda^0(M),\\
\Lambda_7	&\coloneqq \Lambda^1(M),\\
\Lambda_{14}	&\coloneqq \Lambda^2_{14}(M) = \{\beta\in\Lambda^2(M)\,:\, \star(\beta\wedge\ph) = - \beta\}
			=  \{\beta\in\Lambda^2(M)\,:\, \beta\wedge\psi=0\},\\
\Lambda_{27}	&\coloneqq \Lambda^3_{27}(M) = \{\gamma\in\Lambda^3(M)\,:\, \gamma\wedge\ph=0,~\gamma\wedge\psi=0\}. 
\end{align*}
\begin{prop}[\cite{Bryant}]\label{propBryantFormulae}
Let $M$ be a $7$-manifold with a torsion-free $G_2$-structure $\ph$. 
Then, for all $p,q\in\{1,7,14,27\}$ there exist a first-order differential operator $d^p_q:\Lambda_p\to\Lambda_q$ so that the following 
exterior derivative formulae hold for all $f\in\Lambda_1$, $\alpha\in\Lambda_7$, $\beta\in\Lambda_{14}$, $\gamma\in\Lambda_{27}$:
\[\renewcommand{\arraystretch}{1.3}
\begin{array}{lcrrrr} 
df	&	=	&	&	d^1_7f	&	&	\\
d(f\ph)	&	=	&	&	d^1_7f\w\ph	&	&	\\
d(f\psi)	&	=	&	&	d^1_7f\w\psi	&	&	\\
d\alpha	&	=	&	&	\tfrac13\star(d^7_7\alpha\w\psi)	&	+d^7_{14}\alpha	&	\\
d\star(\alpha\w\psi)	&	=	& -\tfrac37\,d^7_1\alpha\cdot\ph	& -\tfrac12\star(d^7_7\alpha\w\ph)		&	& +d^7_{27}\alpha	\\
d\star(\alpha\w\ph)	&	=	& \tfrac47\,d^7_1\alpha\cdot\psi	& +\tfrac12 d^7_7\alpha\w\ph		&	& +\star d^7_{27}\alpha	\\
d(\alpha\w\ph)	&	=	&	&\tfrac23\, d^7_7\alpha\w\psi	& - \star d^7_{14}\alpha	&	\\
d(\alpha\w\psi)	&	=	&	& \star d^7_7\alpha	&	&	\\
d\star\alpha	&	=	&	-d^7_1\alpha \star1 &	&	&	\\
d\beta		&	=	&	&	\tfrac14\star(d^{14}_7\beta\w\ph)&	& +d^{14}_{27}\beta	\\
d\star\beta	&	=	&	&\star d^{14}_7\beta	&		&	\\
d\gamma		&	=	&	&	\frac14 d^{27}_7\gamma\wedge\ph	&	& +\star d^{27}_{27}\gamma	\\
d\star\gamma	&	=	&	& -\frac13 d^{27}_7\gamma\wedge\psi&  -\star d^{27}_{14}\gamma & 
\end{array}
\]
With respect to the natural metrics on the underlying subbundles, $d^p_q = (d^q_p)^*$. 
\end{prop}

\begin{remark}
This result is stated in \cite{Bryant} and it is discussed also in \cite{CKT}. The latter provides a detailed proof, but uses different  
conventions.   

Let us collect here a few comments.  
\begin{enumerate}[$\bullet$]
\item The operators $d^p_q$ are defined by decomposing the differential form in the LHS of each formula into types. For instance
\[
\begin{split}
&d^1_7f \coloneqq d f,\\
&d^7_1\alpha \coloneqq d^* \alpha,\quad d^7_{7}\alpha \coloneqq \star d(\alpha\w\psi), \quad d^7_{14} \alpha \coloneqq \pi^2_{14}(d\alpha),\quad 
d^7_{27}\alpha \coloneqq \pi^3_{27}(d\star(\alpha\w\psi)),\\
&d^{14}_7\beta	\coloneqq d^*\beta,\quad d^{14}_{27}\beta \coloneqq \pi^3_{27}(d\beta),\\
&d^{27}_7\gamma \coloneqq \star(d^*\gamma\wedge\psi),\quad d^{27}_{14}\gamma\coloneqq \pi^{2}_{14}(d^*\gamma),\quad 
d^{27}_{27}\gamma \coloneqq \pi^3_{27}(\star d\gamma).
\end{split}
\]
These definitions imply that $d^p_q = (d^q_p)^*$, where the RHS is the formal adjoint.
\item Some components in the RHS are zero by the defining property of the subspace $\Lambda_p$, e.g.~$\pi^3_1(d\beta) = 0$ since $\beta\wedge\psi=0$, 
and $\pi^4_1(d\gamma) = 0$ since $\gamma\wedge\ph = 0$. 
\item The fact that the same operator (e.g.~$d^7_{7}\alpha$, $d^{27}_7\gamma$) appears in different formulae follows from 
$d^p_q = (d^q_p)^*$ and from the $G_2$ identities  
\begin{equation}\label{G2id1form}
\star(\alpha\w\ph)\w\ph = -4 \star\alpha,\qquad
\star(\alpha\w\psi)\w\psi = 3 \star\alpha,\qquad \forall\alpha\in\Lambda^1(M). 
\end{equation}
For instance, let us consider the 3-form $d\star(\alpha\w\psi)$. It decomposes according to the decomposition of $\Lambda^3(M)$ as follows
\[
d\star(\alpha\w\psi) = u\cdot \ph + \star(\theta\wedge \ph) + \rho,
\]
where $u\in\Lambda_1$, $\theta\in\Lambda_7$ and $\rho\in\Lambda_{27}$. 
By definition $\rho = \pi^3_{27}(d\star(\alpha\w\psi)) \eqqcolon d^7_{27}\alpha$. 
Wedging both sides by $\psi$ and using the second identity in \eqref{G2id1form} we get
\[
7\,u\,\star1 = d\star(\alpha\w\psi) \w \psi = 3 d\star \alpha \quad \Rightarrow \quad u = -\tfrac 37 d^*\alpha = -\tfrac 37 d^7_1\alpha. 
\]
On the other hand, wedging by $\ph$ and using the first identity in \eqref{G2id1form} we obtain
\[
d\star(\alpha\w\psi) \w \ph = \star(\theta\wedge \ph)\w\ph = -4 \star\theta.
\]
Since $\star(\alpha\w\psi) \in \Lambda^2_7(M)$, one has $\star(\alpha\w\psi) \w \ph = 2\, \alpha\w\psi$, thus
\[
\theta = -\tfrac14 \star d(\star(\alpha\w\psi) \w \ph) = -\tfrac12 \star d( \alpha\w\psi) = -\tfrac12 d^7_7\alpha. 
\]
The proof of the remaining identities is similar. 
\item These decompositions are analogous to those seen above for complex manifolds. 
The algebraic operations of wedging with $\ph$ or $\psi$ are analogous to the Lefschetz operator on K\"ahler manifolds. 
\end{enumerate}
\end{remark}

The identity $d^2=0$ is equivalent to certain second order identities on the operators $d^p_q$. Here, we will need the following
\begin{equation}\label{e:d2id}
d^7_{14}d^7_7 + 2\, d^{27}_{14}d^7_{27} = 0, 
\end{equation}
which is a consequence of $d^2\star(\alpha\w\ph)=0$. 

We will also use the fact that, by definition, the operator $d^7_7:\Lambda_7\to\Lambda_7$ satisfies  
\[
d^7_7(df) = \star d (df\w\psi)=0.
\]
We refer the reader to \cite[Table 2]{Bryant} for the complete list of second order identities. 

\smallskip

We conclude this section with some observations on the operator $d^7_{27}$ and its formal adjoint $d^{27}_7$ that will be useful in the sequel.

Let $\gamma\in\Lambda^3_{27}(M)$. From the expression of $d\gamma$ and the first identity in \eqref{G2id1form}, one has
\[
d^{27}_7\gamma = - \star( \star d\gamma \wedge\ph).  
\]
Moreover, by definition 
\[
d^{27}_7\gamma = - \star( \star d\star\gamma\wedge\psi) = \star(d^*\gamma\wedge\psi).
\]
Therefore, for every $\gamma\in\Lambda^3_{27}(M)$ the following identity holds
\begin{equation}\label{DifId327}
d^*\gamma\wedge\psi = \ph \wedge  \star d\gamma.
\end{equation}

\smallskip 

The next result provides some information about the operator $d^7_{27}$.
\begin{lem}\label{l:d727}
Let $(M,\ph)$ be a $7$-manifold with a torsion-free $G_2$-structure. Consider the operator
$d^7_{27}:\Lambda_7\rightarrow\Lambda_{27}$.
\begin{enumerate}[1.]
\item The condition $d d^7_{27}(\alpha)=0$ is equivalent to the two equations
\[
d^{27}_7d^7_{27}(\alpha)=0, \quad d^{27}_{27}d^7_{27}(\alpha)=0.
\]
\item Assume $M$ is compact. Then the first of these equations is equivalent to $d^7_{27}(\alpha)=0$ so it implies the second. In other words:
\[
d d^7_{27}(\alpha)=0\quad\Leftrightarrow\quad d^7_{27}(\alpha)=0.
\]
In particular, the spaces $\Im(d^7_{27})$ and $\Ker(d|_{\Lambda_{27}})$ are orthogonal. 

\item The first equation in 1.~can be written as
\[
\Delta\alpha+\frac57 dd^*\alpha=0.
\]
If $\alpha$ is closed then this equation is equivalent to $dd^*\alpha=0$.
\item Assume $M$ is compact. Then $d^7_{27}(\alpha)=0$ if and only if $\alpha$ is harmonic. 
\end{enumerate}
\end{lem}

\begin{proof}
The first statement follows from Bryant's formula 
\[
dd^7_{27}(\alpha)	= \frac14 d^{27}_{7}d^7_{27}(\alpha) \wedge \ph + \star d^{27}_{27}d^7_{27}(\alpha).
\]
The second statement follows from the fact that $d^{27}_7=(d^7_{27})^*$, so that
\[
\int_M (d^{27}_7d^7_{27}(\alpha),\alpha)\vol_{g_\ph}=\int_M |d^7_{27}(\alpha)|^2 \vol_{g_\ph}.
\]
To prove orthogonality, choose any $\gamma\in \Ker (d)\cap\Lambda^3_{27}(M)$. Then
\begin{align*}
\int_M (d^7_{27}(\alpha),\gamma) \vol_{g_\ph}&=\int_M(\star d^7_{27}(\alpha),\star\gamma)\vol_{g_\ph}\\
&=\int_M (d\star(\alpha\wedge\ph),\star\gamma)\vol_{g_\ph}\\
&=\int_M(d^*(\alpha\wedge\ph),\gamma)\vol_{g_\ph}\\
&=\int_M(\alpha\wedge\ph,d\gamma)\vol_{g_\ph}=0.
\end{align*}

The third statement follows from Bryant's formulae \cite[Table 2 and Table 3]{Bryant}
\[
\Delta\alpha	=	 \left(\left(d^7_7\right)^2+d^1_7 d^7_1\right)\alpha,\qquad 
(d^7_7)^2		=	d^{27}_7d^7_{27} -\frac{12}{7}d^1_7d^7_1. 
\]
Regarding statement 4., we have seen in 2.~that $d^7_{27}(\alpha)=0$ if and only if $d^{27}_7d^7_{27}(\alpha)=0$. 
Using 3., the latter can be rewritten as follows 
\[
0 = \Delta\alpha +\frac57dd^*\alpha = \frac{12}{7}dd^*\alpha + d^*d\alpha,
\]
so that 
\[
0 = \int_M \left(\frac{12}{7}dd^*\alpha + d^*d\alpha, \alpha\right)\vol_{g_\ph} = \frac{12}{7}|d^*\alpha|^2 + |d\alpha|^2. 
\]
Statement 4.~then follows from the fact that $\alpha$ is harmonic if and only if it is closed and coclosed.
\end{proof}

We remark that Statement 4.~does not appear amongst Bryant's formulae because it requires integration by parts, thus compactness.

Summarizing and using standard Hodge theory, we obtain the following conclusion.

\begin{cor}\label{cor:harmonic 1forms}
Let $(M,\ph)$ be a compact $7$-manifold with a torsion-free $G_2$-structure. Choose $\alpha\in\Lambda^1(M)$. Then
$\alpha$ is harmonic if and only if $d^7_{27}(\alpha)=0$.

In particular, assume $\alpha=df$. Then $d^7_{27}(df)=0$ if and only if $df=0$, i.e., $f$ is constant.
\end{cor}

\subsection{Harmonic forms.}
Let $(M,\ph)$ be 7-manifold with a torsion-free $G_2$-structure. 
As in the K\"ahler case, but for a different reason, the Hodge Laplacian operator $\Delta$ preserves the decomposition of $\Lambda^k(M)$ 
into types. This follows either from the general theory of torsion-free Riemannian $G$-structures explained by Joyce \cite[Sect.~3.5.2]{Joyce}, 
or by Bryant-Harvey's formulae \cite[Table 3]{Bryant}. Let $\mathcal{H}^k_q$ denote the harmonic forms in $\Lambda^k_q$. 
Then, when $M$ is a compact, one obtains decompositions such as
\[
\mathcal{H}^2=\mathcal{H}^2_7\oplus\mathcal{H}^2_{14}, \quad \mathcal{H}^3=\mathcal{H}^3_1\oplus\mathcal{H}^3_7\oplus\mathcal{H}^3_{27},
\]
etc. 
However, analogously to the refined $\U(n)$-decomposition in the K\"ahler case, there is no cohomological viewpoint on these spaces.

\subsection{The global $dd^\ph$ lemma}

In order to characterize $\Im(dd^\ph)$, it is important to notice that forms of type $dd^\ph f$ have some special features. 
To start, we observe that the general identity 
\begin{equation}\label{eqalphahookphi}
\alpha^\sharp\lrcorner\ph = \star(\alpha\w\psi), 
\end{equation}
which holds for every $\alpha\in\Lambda^1(M)$, 
allows one to rewrite
\[
dd^\ph f = d(\nabla f\lrcorner\ph) = d\star(df\w\psi). 
\] 

\begin{lem}\label{l:ddphi_features}
Let $(M,\ph)$ be a $7$-manifold with a torsion-free $G_2$-structure.  
Given $f\in\Lambda^0(M)$, consider the $3$-form $dd^\ph f$. It has the following properties:
\begin{enumerate}
\item $dd^\ph f$ is exact;
\item $dd^\ph f\in\Lambda^3_1(M)\oplus\Lambda^3_{27}(M)$ and 
\[
dd^\ph f=-\frac37\Delta f\cdot\ph+ d^7_{27}(df); 
\]
\item $\pi^2_{14}d^*(dd^\ph f)=0$;
\item Assume $M$ is compact. Then $\pi^3_1(dd^\ph f) =0$ if and only if $\pi^3_{27}(dd^\ph f) =0$. In other words, $dd^\ph f$ is not of pure type 
(unless it vanishes).
\end{enumerate}
\end{lem}
\begin{proof} Statement 1.~is obvious. Regarding Statement 2., Bryant's formulae (Proposition \ref{propBryantFormulae}) and the identity 
$d^7_7(df)=0$ give 
\[
dd^\ph f = d\star(df\w\psi) = -\frac37\Delta f\cdot\ph+d^7_{27}(df), 
\]
showing, in particular, that $dd^\ph f\in\Lambda^3_1(M)\oplus\Lambda^3_{27}(M)$. 
Note that this last fact also follows from general theory, since $dd^\ph f = \lambda_\ph(\nabla\nabla f)$. 

Let us now prove Statement 3. Again using Bryant's formulae we find that, for any $u\in \Lambda^0(M)$, $d^*(u\cdot\ph)$ has no component in 
$\Lambda^2_{14}(M)$. 
Furthermore, from \eqref{e:d2id} we obtain
\[
\pi^2_{14}d^*(d^7_{27}(df))=d^{27}_{14}d^7_{27}(df) = -\frac12 d^7_{14}d^7_7(df) = 0.
\]
Finally, Statement 4.~follows from Corollary \ref{cor:harmonic 1forms}.
\end{proof}

\begin{remark}\label{RemExtraHL}
An alternative proof of Statement 3.~can be obtained as follows. Using Statement 2., Proposition \ref{propBryantFormulae}, 
and $d^7_7(df)=0$
\[
\begin{split}
\star(dd^\ph f)    & =-\frac{3}{7}\Delta f\cdot \psi+\star d^7_{27}(df)\\
                    & =-\Delta f\cdot\psi+\frac{4}{7}\Delta f\cdot\psi+\star d^7_{27}(df)\\
                    & =-\Delta f\cdot\psi+d\star(df\wedge\ph).
\end{split}
\]
It follows that $d^*(dd^\ph f)) = -\star d\star (dd^\ph f) =\star(d(\Delta f)\wedge \psi))\in\Lambda^2_{7}$, showing that the $\Lambda^2_{14}$ component is zero.

Statement 4. does not appear in Harvey-Lawson's papers because they do not discuss decomposition into types.
\end{remark}

We can formalize this situation via the following spaces.

\begin{definition}
Let $(M,\ph)$ be $7$-manifold with a torsion-free $G_2$-structure. Set
\[
C\coloneqq \left(\Lambda^3_1(M)\oplus\Lambda^3_{27}(M)\right)\cap \Ker(\pi^2_{14}d^*), \qquad K \coloneqq C\cap \Ker(d).
\]
\end{definition}

According to Lemma \ref{l:ddphi_features}, $\Im(dd^\ph)\subseteq K$. 
We now want to characterize its complement. This requires understanding the analytic properties of $dd^\ph$.

\begin{lem}\label{l:G2injectivity} 
The principal symbol $\sigma(dd^\ph)$ 
of the linear differential operator $dd^\ph:\Lambda^0(M)\rightarrow\Lambda^3_1(M)\oplus\Lambda^3_{27}(M)$ gives the homomorphisms
\[ 
\sigma(dd^\ph)|_{x}(\xi):\R\to\Lambda^3_1(T^*_xM)\oplus\Lambda^3_{27}(T^*_xM) ,\quad 
\sigma(dd^\ph)|_{x}(\xi) : c \mapsto c\,\xi \wedge (\xi^\sharp\lrcorner \ph),
\]
for all $x\in M$ and $\xi\in T^*_xM$. 
This symbol is injective.
\end{lem}
\begin{proof}
Choose a $G_2$ coframe $(e^1,\ldots,e^7)$ of $T^*_xM$, so that 
\[
\ph|_x = e^{123} + e^{145} + e^{167} + e^{246} - e^{257} - e^{347} - e^{356}.
\]
Consider a non-zero $\xi \in T^*_xM$. 
Since $G_2$ acts transitively on the unit sphere in $T_x^*M$ preserving $\ph|_x$, 
it is not restrictive to assume that $\xi = a e^1,$ for a certain $a\neq0$. We then have 
\[
\xi \wedge (\xi^\sharp\lrcorner \ph) = a^2 \left(e^{123}+e^{145}+e^{167}\right). 
\]
From this, it easily follows that the map $c\mapsto c\, \xi \wedge (\xi^\sharp\lrcorner \ph)$ is injective. 
\end{proof}

We are now ready to prove the main result of this section, which represents the analogue of Theorem \ref{ThmSplittingKahler} 
for torsion-free $G_2$ 7-manifolds.
The technical details in the proof are similar to those seen in the K\"ahler case, 
so we will provide only a sketch of the main arguments.

\begin{thm}\label{thm:G2Hodge}
Let $(M,\ph)$ be a compact $7$-manifold with a torsion-free $G_2$-structure. Then there exists an $L^2$-orthogonal decomposition 
\[
K=\Im(dd^\ph)\oplus\mathcal{H}^3_{1}\oplus\mathcal{H}^3_{27}.
\]
\end{thm}

\begin{proof}
Consider the operator
\[
dd^\ph:\Lambda^0(M)\rightarrow\Lambda^3_1(M)\oplus\Lambda^3_{27}(M),  
\]
and a generic element $u\cdot\ph+\gamma \in \Lambda^3_1(M)\oplus\Lambda^3_{27}(M)$. 
Endow both spaces with the $L^2$ metric. Notice that
\[
\begin{split}
\int_M (dd^\ph f, u\cdot\ph+\gamma )\vol_{g_\ph}
&=\int_M (\nabla f\lrcorner\ph,d^*(u\cdot\ph+\gamma))\vol_{g_\ph}\\
&=\int_M (\star(df\w\psi),d^*(u\cdot\ph+\gamma))\vol_{g_\ph}\\
&=\int_Mdf\wedge\psi\wedge d^*(u\cdot\ph+\gamma)\\
&=-\int_Mf\, \psi\wedge dd^*(u\cdot\ph+\gamma). 
\end{split}
\]
It follows that, within the given spaces, the adjoint of $dd^\ph$ is the map
\[
(dd^\ph)^t = - \star(dd^*(\cdot)\wedge\psi) : \Lambda^3_1(M)\oplus\Lambda^3_{27}(M) \to \Lambda^0(M). 
\]

Consider now the unbounded operator $dd^\ph:L^2(M)\to L^2(\Lambda^3_1(M)\oplus\Lambda^3_{27}(M))$ with domain $D=H^2(M)$. 
Its adjoint $(dd^\ph)^*$ is an appropriate extension of $(dd^\ph)^t$. 
General theory (cf.~Proposition \ref{p:abstractdecomp}) and Lemma \ref{l:G2injectivity} lead to the decomposition
\[
L^2(\Lambda^3_1(M)\oplus\Lambda^3_{27}(M))=\Im(dd^\ph)\oplus \Ker((dd^\ph)^*). 
\]

We now want to refine this setup, studying the operator $dd^\ph:\Lambda^0(M)\rightarrow K$. To this end, we consider the operators 
\[
d:\Lambda^3(M)\rightarrow\Lambda^4(M), \quad \pi^2_{14}d^*:\Lambda^3(M)\rightarrow\Lambda^2_{14}(M),
\]
and the auxiliary operator $Q:\Lambda^3_1(M)\oplus\Lambda^3_{27}(M)\rightarrow\Lambda^0(M)\oplus\Lambda^4(M)\oplus\Lambda^2_{14}(M)$
\[
Q:\alpha\mapsto\left((dd^\ph)^t(\alpha),d\alpha,\pi^2_{14}d^*(\alpha)\right).
\]
Its principal symbol $\sigma(Q)$ gives the following homomorphism at $x\in M$ 

\[
\sigma(Q)|_x(\xi): \alpha \mapsto \left( \star (\xi \wedge(\xi^\sharp\lrcorner\alpha)\wedge\psi), \xi\wedge\alpha, -\pi^2_{14}(\xi^\sharp \lrcorner \alpha)\right),
\]
for every $\xi\in T^*_xM$. We now prove that the principal symbol is injective. 

Let $\xi\neq0$ and consider $\alpha\in \Ker(\sigma(Q)|_x(\xi))\subseteq \Lambda^3_1(T^*_xM)\oplus\Lambda^3_{27}(T^*_xM)$. 
Using $\xi\wedge\alpha=0$ we obtain
\[
0 =  \star (\xi \wedge(\xi^\sharp\lrcorner\alpha)\wedge\psi) = |\xi|^2 \star(\alpha\wedge\psi),
\]
which implies $\alpha\in\Lambda^3_{27}(T^*_xM)$. 
We claim that $\xi^\sharp\lrcorner \alpha \in \Lambda^2_{14}(T^*_xM)$: this can be proved in two ways. The first way is to choose a $G_2$ coframe $(e^1,\ldots,e^7)$ of $T^*_xM$ and use the $G_2$ action to reduce to the case $\xi = a e^1$ with $a\neq0$. 
Write $\alpha\in\Lambda^3_{27}$ in terms of its $27$ independent coefficients. The condition $0 = \xi\w\alpha =a\,e^1\wedge\alpha$ 
gives twenty linear equations, corresponding to the coefficients of the wedge products $e^{1ijk}$, where $2\leq i<j<k\leq7$. 
Substituting their solutions into $(\xi\lrcorner\alpha)\wedge\psi$ shows that this expression vanishes,  
hence $\xi^\sharp\lrcorner \alpha \in \Lambda^2_{14}(T^*_xM)$. Alternatively, the claim can be obtained from the $G_2$ identity
\[
\varphi\wedge\star(\xi\wedge\alpha)=-(\xi^\sharp\lrcorner\alpha)\wedge\psi,
\qquad \alpha\in\Lambda^3_{27}, 
\]
which follows from $G_2$ representation theory. 

On the other hand, $\alpha\in \Ker(\sigma(Q)|_x(\xi))$ implies $\pi^2_{14}(\xi^\sharp \lrcorner \alpha)=0$. Comparing this to the above claim, we find $\xi^\sharp\lrcorner \alpha = 0$.  
This in turn implies  $|\xi|^2\alpha = \xi \wedge (\xi^\sharp \lrcorner \alpha) =0$,
whence $\alpha=0$. This concludes the proof that the principal symbol is injective. Now notice:
\[
\Ker\left((dd^\ph)^*|_{\overline{\Ker(d)\cap\Ker(\pi^2_{14}d^*)}}\right) = \Ker\left((dd^\ph)^t|_{\Ker(d)\cap\Ker(\pi^2_{14}d^*)}\right)=\Ker((dd^\ph)^t|_{K}).
\]
We want to characterize $\Ker((dd^\ph)^t|_{K})$. As a first step, notice that 
\[
(dd^\ph)^t|_{K}=-\star(\Delta|_{K}(\cdot)\wedge\psi).
\]
It follows that
\[
\Ker((dd^\ph)^t|_{K})=\Ker(\Delta(\cdot)\wedge\psi)\cap K,
\]
where $\Delta$ has domain $\Lambda^3_1(M)\oplus\Lambda^3_{27}(M)$.
Recall that $\Delta$ preserves types.
Recall also that, for all $\gamma\in\Lambda^3_{27}(M)$, $\gamma\wedge\psi=0$. Then
\[
\Delta(u\cdot\ph +\gamma)\w\psi = \Delta(u\cdot\ph)\w\psi + \Delta\gamma\w\psi  = 7\Delta u \star1. 
\]
Since $\R\cdot \ph \leq K$, we obtain 
\[
\Ker(\Delta(\cdot)\wedge\psi)\cap K = {\left(\R\cdot\ph\oplus\Lambda^3_{27}(M)\right)}\cap K  = \R\cdot\ph\oplus(\Lambda^3_{27}(M)\cap K).
\]
Now,  Bryant's formula for $d\gamma$ shows that
\[
d\gamma=0 \quad\Leftrightarrow\quad d^{27}_7\gamma=0\quad \mbox{and} \quad d^{27}_{27}\gamma=0.
\]
Comparing with the formula for $\d^*\gamma=0$, we find that the additional condition $\pi^2_{14}(d^*\gamma)=0$ implies that  
$\gamma\in\mathcal{H}^3_{27}$, so that $\Lambda^3_{27}(M)\cap K = \mathcal{H}^3_{27}$.
This shows that
$\Ker((dd^\ph)^t|_{K})=\R\cdot\ph\oplus\mathcal{H}^3_{27}=\mathcal{H}^3_{1}\oplus \mathcal{H}^3_{27}$.

Putting everything together, we learn that
\[
L^2(\Lambda^3_1(M)\oplus\Lambda^3_{27}(M))\cap\overline{\Ker(d)\cap\Ker(\pi^2_{14}d^*)} = 
\Im(dd^\ph)\oplus\mathcal{H}^3_{1}\oplus \mathcal{H}^3_{27}.
\]
Intersecting both sides with $K$, we obtain
$$K=\Im(dd^\ph)\oplus\mathcal{H}^3_{1}\oplus \mathcal{H}^3_{27}.$$
\end{proof}

\begin{remark}
Notice the analogies with the K\"ahler case: $\Lambda^3_{27}(M)$ corresponds to $\Lambda^{1,1}_0$, 
$\Lambda^3_{1}(M)\oplus\Lambda^3_{27}(M)$ corresponds to $\Lambda^{1,1}_\R$, 
and the extra condition $\gamma\in \Ker(\pi^2_{14}d^*)$ replaces the linear algebra formula for $\star\alpha$. Alternatively, $\pi^2_{14}d^*$ plays the role of $d^c$. 
Notice, however, that $\ker(d^c)$ is trivial when $d^c$ is restricted to $\Lambda^{1,1}_\R\cap\ker(d)$, so it does not appear in the discussion 
in Section \ref{sec:cohomologyBC}.
\end{remark}

Using Theorem \ref{thm:G2Hodge} we obtain the following analogue of the global $\partial\bar\partial$ (or $dd^\omega$) lemma.

\begin{cor}[Global $dd^\ph$ lemma]\label{cor:G2globallemma} 
Let $(M,\ph)$ be a compact $7$-manifold with a torsion-free $G_2$-structure. 
Fix a form $\alpha\in\Lambda^3_{1}(M)\oplus\Lambda^3_{27}(M)$. 
If (i) $\alpha$ is globally $d$-exact and (ii) $\pi^2_{14}d^*(\alpha)=0$, then $\alpha$ is globally $dd^\ph$-exact.
\end{cor}

\begin{proof}
The proof is similar to the K\"ahler case. Since the $G_2$-structure is torsion-free, we have 
$\mathcal{H}^3(M)\cap \left(\Lambda^3_1(M)\oplus\Lambda^3_{27}(M)\right) = \mathcal{H}^3_{1}\oplus \mathcal{H}^3_{27}$. Consider $d:\Lambda^2(M)\rightarrow\Lambda^3(M)$. Since $\Im(d)$ is orthogonal to harmonic forms, $\Im(d)\cap  \left(\Lambda^3_1(M)\oplus\Lambda^3_{27}(M)\right)$ 
is orthogonal to $\mathcal{H}^3_{1}\oplus \mathcal{H}^3_{27}$. Theorem \ref{thm:G2Hodge} thus shows that $\Im(d)$ is contained in $\Im(dd^\ph)$. 
The opposite inclusion follows from Lemma \ref{l:ddphi_features}. 
\end{proof}

\begin{remark}
Since $\ph\in\mathcal{H}^3_1$, Theorem \ref{thm:G2Hodge} shows that it is not globally $dd^\ph$ exact when $M$ is compact. 
However, this is obvious since the 3-form defining a torsion-free $G_2$-structure cannot be $d$-exact. 
We will see in Section \ref{s:cddlocal} that $\ph$ is not even locally $dd^\ph$-exact. 
  
The situation is different in the non-compact setting. For instance, the standard flat torsion-free $G_2$-structure $\ph_o$ on $\R^7$ is globally $dd^{\ph_o}$-exact:
\[
\ph_o(x) = dd^{\ph_o}\left(\tfrac16|x|^2\right),
\]
for all $x\in\R^7$.  
\end{remark}

\paragraph{{Alternative characterization.}}Theorem \ref{thm:G2Hodge} provides a characterization of the space $\Im(dd^\ph)$ as the space orthogonal to certain harmonic forms. Using Bryant's formulae and Lemma \ref{l:d727} we can provide an alternative characterization. 
Recall the expression
\[
dd^\ph f=-\frac37\Delta f\cdot \ph+d^7_{27}(df).
\]
Recall also that, when $M$ is compact, $\int_M\Delta f \vol_{g_\ph}=0$. It thus makes sense to consider the space
\[
X\coloneqq \left\{
u\cdot\ph+\gamma\in (\Lambda^3_{1}\oplus\Lambda^3_{27})\cap \Ker(d)~:~ 
(i) \int_Mu\vol_{g_\ph}=0,~(ii)~\gamma \in \Im(d^7_{27})
\right\}.
\]

\begin{prop}
Assume $(M,\ph)$ is a compact $7$-manifold with a torsion-free $G_2$-structure. Then $X=\Im(dd^\ph)$. In particular, all forms in $X$ are exact.
\end{prop}
\begin{proof}
Clearly, $X$ contains all forms of type $dd^\ph f$. 

Conversely, choose $u\cdot\ph+\gamma\in X$. The closedness condition is equivalent to the equation $d\gamma=-du\wedge\ph$. 
For any given $u$, this is a linear non-homogeneous equation on $\gamma$. 
The fact $\int_Mu\vol_{g_\ph}=0$ is equivalent to the solvability of the equation $\Delta f=-(7/3)u$. 
In this case one can show that the above equation for $\gamma$ admits the solution $\gamma=d^7_{27}(df)$. 
The space of solutions is parametrized by the solutions of the corresponding homogeneous equation $d\gamma=0$. 
Lemma \ref{l:d727} shows that the solution is unique.
\end{proof}

\subsection{{Cohomology}}

The above results allow us to find a cohomological point of view on some of the spaces of harmonic 3-forms, as follows.

Let $(M,\ph)$ be a compact torsion-free $G_2$ 7-manifold.  The sequence 
\begin{equation}\label{eq:G2complex}
\Lambda^0(M) \xrightarrow[]{dd^\varphi} \Lambda^3_1(M)\oplus \Lambda^3_{27}(M) \xrightarrow[]{(d,\pi^2_{14}d^*)} 
\Lambda^4(M)\oplus \Lambda^2_{14}(M)
\end{equation}
defines a complex by Lemma \ref{l:ddphi_features}, since $(d,\pi^2_{14}d^*)\circ dd^\ph=0$. 
This allows us to define the $dd^\ph$-cohomology space 
\[
H^\ph(M)=(\ker(d)\cap\ker(\pi^2_{14}d^*))/\Im(dd^\ph) = K/\Im(dd^\ph). 
\]
This is clearly analogous to Bott-Chern cohomology. 

We can now summarize Theorem \ref{thm:G2Hodge} as follows. 
\begin{cor}\label{cor:G2coh} 
The linear map
\[
H^\ph(M)\to H^3(M;\R),\quad [\alpha]_\ph \mapsto [\alpha]
\]
is an injection which identifies $H^\ph(M)$ with the subspace of de Rham classes represented 
by the real harmonic forms $\mathcal{H}^3_{1} \oplus \mathcal{H}^3_{27}.$
\end{cor}
This provides a cohomological interpretation of $\mathcal{H}^3_{1} \oplus \mathcal{H}^3_{27}$, 
analogous to the interpretation of $\mathcal{H}^{1,1}_\R$ in terms of $H^{1,1}_{BC}(M;\R)$.

Notice that, in general, the $dd^\ph$-cohomology depends on the specific $G_2$ structure. 
If however $\mathrm{Hol}(g_\ph) = G_2$ then $\mathcal{H}^3_7=0$ \cite{Joyce}, so the above map is an isomorphism. 
In other words, in this case $dd^\ph$-cohomology is isomorphic to $H^3(M)$. In this sense it is independent of $\ph$.

\begin{remark}
If instead we introduce the space
$
\widetilde{\Lambda}^3_1(M)\coloneqq \left\{u\cdot\ph~:~\int_M u \vol_{g_\ph}=0\right\} \subset \Lambda^3_1(M), 
$
we can analogously consider the sequence
\[
\Lambda^0(M)\xrightarrow{dd^\ph} \widetilde{\Lambda}^3_1(M)\oplus \Lambda^3_{27}(M)  \xrightarrow[]{(d,\pi^2_{14}d^*)} \Lambda^4(M)\oplus\Lambda^2_{14}(M),
\]
and obtain a cohomological point of view on $\mathcal{H}^3_{27}$ in terms of a reduced $dd^\ph$-cohomology space:
\[
H^\ph_0(M) \coloneqq (\ker(d)\cap\ker(\pi^2_{14}d^*))/\Im(dd^\ph)\simeq \mathcal{H}^3_{27}. 
\]
\end{remark}

\begin{remark}\label{rem:alternativeG2}
As in Remark \ref{rem:alternativeKahler}, an alternative proof of Theorem \ref{thm:G2Hodge} would start by showing that the complex \eqref{eq:G2complex} 
is elliptic at its middle term: for every $x\in M$ and $\xi\in T^*_xM\smallsetminus\{0\}$ we have
\[
\Im(\sigma(dd^\ph)|_x(\xi)) 	= \R\cdot \xi \wedge (\xi^\sharp\lrcorner \ph)  
						= \ker(\sigma(d)|_x(\xi)) \cap \ker(\sigma(\pi^2_{14}d^*)|_x(\xi)).
\]
The proof is similar to those seen above. We then obtain a fourth order elliptic, self-adjoint, linear differential operator 
\[
\Box \coloneqq (dd^\ph)(dd^\ph)^* + (R^*R)^2 : \Lambda^3_1(M)\oplus \Lambda^3_{27}(M) \to \Lambda^3_1(M)\oplus \Lambda^3_{27}(M), 
\]
where $R\coloneqq (d,\pi^2_{14}d^*)$. This leads to the $L^2$-orthogonal splitting $\Lambda^3_1(M)\oplus \Lambda^3_{27}(M) = \Im(\Box)\oplus\ker(\Box)$. 
The statement of Theorem \ref{thm:G2Hodge} then follows after intersecting both sides of this identity with $\ker(d)\cap\ker(\pi^2_{14}d^*)$. 
\end{remark}

\subsection{A counterexample to the local $dd^\ph$ lemma}\label{s:cddlocal}

We have already argued, in the complex case, the very different nature of the global, vs.~the local, $\partial\bar\partial$ lemma.
That said, it is interesting to speculate whether in the $G_2$ case there exists a local version of the $dd^\ph$ lemma. 

Of course, this requires formulating a candidate statement. 
The most obvious such statement would be: given any $7$-manifold $M$ with a torsion-free $G_2$-structure $\ph$, 
any form $\alpha\in K$ is locally $dd^\ph$-exact.

To investigate this, let us assume $M$ is compact. 
Corollary \ref{cor:G2globallemma} shows that any exact form in $K$ is globally of type $dd^\ph(f)$, so it remains to study the harmonic ones. 
Let us focus on $\alpha=\ph$ and assume that locally
\[
\ph=dd^\ph f  =\lambda_\ph(\nabla\nabla f)= \frac12 \i_\ph(\Hess_{g_\ph}(f)).
\]
This implies $g_\ph = 3 \Hess_{g_\ph}(f)$, i.e., $(M,g_\ph)$ is of the following type. 

\begin{definition}[\cite{AmariArmstrong}]
A Riemannian manifold $(M,g)$ is of {\em Hessian type} if locally $g = \mathrm{Hess}_g (f)$, for some smooth function $f$. 
\end{definition}

Amari and Armstrong \cite{AmariArmstrong} prove that any such manifold has vanishing Pontryagin classes.  
On the other hand, if $M$ is compact and  $\mathrm{Hol}(g_\ph) = G_2$ (or, more generally, $g_\ph$ is not flat), 
Joyce \cite{Joyce} shows that $p_1(M)\neq 0$. 
We conclude that on such manifolds $\ph$ is not locally of type $dd^\ph f$. 

The form $\ph$ thus represents a counterexample to the above statement in the compact holonomy $G_2$ case. 

\begin{remark}
This result was, in some sense, to be expected. 
In Section \ref{s:dedebar} we discussed the fact that the local $\partial\bar\partial$ lemma relies on the existence of local pluriharmonic functions. 
In this case, the analogous functions satisfy the condition $dd^\ph f=0$. 
As shown in \cite{HL:intropotential}, such functions are very scarce on torsion-free $G_2$ manifolds.
\end{remark}

\begin{remark}
When $(M,g)$ is compact it is impossible that $g=\Hess_g(f)$ globally: 
indeed, such $f$ would be strictly convex so it could not admit a maximum point. 
This underlines the importance of defining Hessian type in terms of local conditions. 
\end{remark}

\begin{remark}
Above, we study the case $\alpha=\varphi$ which is harmonic in $\Lambda^3_1(M)$. One might wonder what happens for harmonic forms in $\Lambda^3_{27}(M)$. At least in the non-compact case, it is simple to find examples which are not $dd^\varphi$-exact. 
For instance, on $\mathbb{R}^7$ with the standard flat $G_2$ structure $\varphi_o$, the form
\[
\alpha=x^1(e^{137}+e^{234})+x^2(e^{134}-e^{237}), \qquad e^i = dx^i,
\]
belongs to $\Lambda^3_{27}$ and is closed and coclosed, thus harmonic, but is not $dd^{\varphi_o}$-exact on any nonempty open set. 
Indeed, $\alpha = \mathrm{i}_{\varphi_o}(h)$, where $h$ is trace-free symmetric 2-tensor
\[
h = \frac{x_1}{4}(e^1e^4-e^2e^7) - \frac{x_2}{4}(e^1e^7 + e^2e^4).  
\]
If $\alpha = dd^{\varphi_o}(f) = \frac12 \mathrm{i}_\varphi(\mathrm{Hess}_{g_{\varphi_o}}(f))$, for some smooth function $f$ defined on some nonempty open 
set of $\R^7$, then 
$\mathrm{Hess}_{g_{\varphi_o}}(f) = 2h$, which is not possible since 
\[
\frac{\partial}{\partial x^1} h_{1,4}  = \frac14 \neq  \frac{\partial}{\partial x^4} h_{1,1}. 
\] 
\end{remark}

\subsection{The calibration $\star\ph$.}
A torsion-free $G_2$ 7-manifold $(M,\ph)$ has a second natural calibration given by the parallel $4$-form $\psi = \star\ph$ \cite{HL:calibratedgeometry}. 
We can thus consider the operator 
\[
dd^\psi:\Lambda^0(M)\to \Lambda^4(M),\quad dd^\psi(f) = d(\nabla f\lrcorner \psi) = -d\star(df\wedge \ph), 
\]
and use it to obtain the analogues of Theorem \ref{thm:G2Hodge} and corollaries \ref{cor:G2globallemma} and \ref{cor:G2coh} for 
differential $4$-forms. The discussion is similar, so we only provide a sketch of the main arguments.  

First, for every $f\in\Lambda^0(M)$,
\begin{itemize}
\item $dd^\psi f\in\Lambda^4_1(M)\oplus\Lambda^4_{27}(M) = \star\Lambda^3_1(M)\oplus \star\Lambda^3_{27}(M)$;
\item $d^*(dd^\psi f) = -\tfrac17\star(dd^*df\w\ph) \in \Lambda^3_7(M)$. This follows from Bryant-Harvey's formulae and shows that 
$\pi^3_1\oplus\pi^3_{27} (d^*(dd^\psi f)) = 0$. 
\end{itemize}
Arguing as in the proof of Lemma \ref{l:G2injectivity} it can be shown that the principal symbol of $dd^\psi$ is injective.
We can then introduce the complex
\[
\Lambda^0(M)	\xrightarrow[]{dd^\psi} \Lambda^4_1(M)\oplus\Lambda^4_{27}(M) 
			\xrightarrow[]{(d,(\pi^3_1\oplus\pi^3_{27})\circ d^*)} \Lambda^5(M)\oplus \Lambda^3_1(M) \oplus \Lambda^3_{27}(M).
\]
The $dd^\psi$-cohomology space is defined as follows
\[
H^\psi(M) \coloneqq (\ker(d) \cap \ker((\pi^3_1\oplus\pi^3_{27})\circ d^* )/\Im(dd^\psi).
\]
The analogue of Theorem \ref{thm:G2Hodge} in this setting shows the existence of an injective linear map  
\[
H^\psi(M)\to H^4(M),\quad [\alpha]_\psi \mapsto [\alpha], 
\]
which defines an isomorphism 
$H^\psi(M)\simeq \mathcal{H}^4_{1} \oplus \mathcal{H}^4_{27}.$ The Hodge operator provides an isomorphism between the $dd^\ph$-cohomology space $H^\ph(M)$ and the 
$dd^\psi$-cohomology space $H^\psi(M)$.

\section{Calabi-Yau 6-manifolds}\label{Sec:CY6}
Let $(N,h,J,\omega,\Omega)$ be a Calabi-Yau 6-manifold: $J$ is a complex structure, 
$h$ is a K\"ahler metric with K\"ahler form $\omega = h(J\cdot,\cdot)$,  
and $\Omega = \Omega^++i\Omega^-$ is a holomorphic complex volume form. 
In particular, $\Omega^+$ is a parallel calibration on the Riemannian manifold $(N,h)$ \cite{HL:calibratedgeometry}, so one can consider the operator $dd^{\Omega^+}$. 

It is well-known that the Calabi-Yau SU(3)-structure on $N$ induces a torsion-free $G_2$-structure on $M=N\times S^1$ defined by the 3-form
\[
\ph = \omega \wedge \eta + \Omega^+, 
\] 
where $\eta\in\Lambda^1(T^*S^1)$, $|\eta|=1$. Let $f\in \Lambda^0(N)$. Then 
\[
\begin{split}
dd^\ph f	&=	d(\nabla^{g_\phi} f \lrcorner \ph) = d(\nabla^h f \lrcorner \omega) \wedge \eta + d(\nabla^h f \lrcorner \Omega^+)\\
		&= dd^\omega f \wedge \eta + dd^{\Omega^+}f,
\end{split}
\]
showing how the operator $dd^\ph$ on $M$ is related to the $dd^\omega$ operator and to the operator $dd^{\Omega^+}$ on $N.$ 

In this section, we investigate the properties of the operator $dd^{\Omega^+}$ and we prove a global $dd^{\Omega^+}$ lemma. 
As we will see, the discussion shares many similarities with the $G_2$ case.

\subsection{Linear algebra} 
Let $(V,h,J,\omega,\Omega = \Omega^++i\Omega^-)$ be an SU(3) vector space. 
Recall the identities $\omega = h(J\cdot,\cdot)$, $\Omega^- = J^*\Omega^+ = \star \Omega^+$, $\omega\w\Omega^\pm=0$, 
and $2\omega^3=3\Omega^+\w\Omega^-$. 
Recall also the decompositions into SU(3)-irreducible subspaces
\[
\begin{split}
\Lambda^2V^*	&= \R \cdot \omega \oplus \Lambda^2_6 \oplus \Lambda^2_8,\\
\Lambda^3V^*	&= \R \cdot \Omega^+ \oplus \R\cdot \Omega^- \oplus \Lambda^3_6 \oplus \Lambda^3_{12}, 
\end{split}
\]
where 
\[
\begin{split}
\Lambda^2_6	&= \left\{\star(\alpha\wedge\Omega^+) ~:~  \alpha\in \Lambda^1V^* \right\}, \\ 
\Lambda^2_8	&= \left\{\sigma\in\Lambda^2V^* ~:~  J^*\sigma=\sigma \mbox{ and } \sigma\wedge\omega^2=0\right\} \cong \mathfrak{su}(3),  
\end{split}
\]
and 
\[
\begin{split}
\Lambda^3_6	&= \left\{\alpha \wedge \omega ~:~  \alpha\in \Lambda^1V^*\right\}, \\
\Lambda^3_{12}	&= \left\{\rho\in\Lambda^3V^* ~:~  \rho\wedge\omega=0 \mbox{ and } \rho\wedge\Omega^{\pm}=0\right\}. 
\end{split}
\]

As in the U$(h)$ case (cf.~Section \ref{SecKahlerLA}), the space of symmetric endomorphisms decomposes into SU(3)-irreducible subspaces as follows
\[
\mathrm{Sym}(V) = \R\cdot \mathrm{Id} \oplus \mathrm{Sym}^{+}_{0} \oplus  \mathrm{Sym}^{-}. 
\]
The following isomorphisms of SU(3)-representations occur:
\[
\begin{split}
\mathrm{Sym}^{+}_{0} \rightarrow\Lambda^2_8,&\quad  A \mapsto h(AJ\cdot, \cdot), \\
\mathrm{Sym}^{-} \rightarrow \Lambda^3_{12},	&\quad  S \mapsto S\cdot\Omega^+ = \Omega^+(S\cdot,\cdot,\cdot) +\Omega^+(\cdot,S\cdot,\cdot) 
+\Omega^+(\cdot,\cdot,S\cdot). 
\end{split}
\]

The restriction of the map $\lambda_{\Omega^+}:\mathrm{End}(V)\to\Lambda^3V^*$ to $\mathrm{Sym}(V)$ 
has kernel $\mathrm{Sym}^{+}_{0}$ and image $\R\cdot \Omega^+ \oplus \Lambda^3_{12}$. We then have the following. 

\begin{lem}
The map $\i_{\Omega^+} \coloneqq \lambda_{\Omega^+}|_{\R\, \mathrm{Id} \oplus  \mathrm{Sym}^{-}}$ defines an isomorphism 
\[
\i_{\Omega^+}: \R\cdot \mathrm{Id} \oplus  \mathrm{Sym}^{-} \to \R\cdot \Omega^+ \oplus \Lambda^3_{12}, \quad \i_{\Omega^+}(A) = A\cdot\Omega^+. 
\]
\end{lem}

Finally, we recall the following well-known SU(3)-identities. 
\begin{lem}\label{l:SU3ids}
For every $\alpha\in\Lambda^1V^*$, the following identities hold 
\begin{enumerate}[i)]
\item $\star(\alpha\w\omega) = -J^*\alpha\w\omega$;
\item $\star(\alpha\wedge\Omega^\pm)\wedge\omega =\alpha\wedge\Omega^\pm = \mp J^*\alpha\wedge\Omega^\mp$;
\item $\star(\alpha\wedge\Omega^\pm)\wedge\omega^2 = 0$;
\item $\star(\alpha\wedge\Omega^-)\wedge\Omega^+ = -\star(\alpha\wedge\Omega^+)\wedge\Omega^- =  \alpha\wedge\omega^2=2\star(J^*\alpha)$;
\item $\star(\alpha\wedge\Omega^-)\wedge\Omega^- = \star(\alpha\wedge\Omega^+)\wedge\Omega^+  = -J^*\alpha\wedge\omega^2=2\star\alpha$.
\end{enumerate}
\end{lem}

\subsection{The decomposition of $d$}
Let $(N,h,J,\omega,\Omega)$ be a Calabi-Yau 6-manifold.  
As in the complex and in the $G_2$ case, the decompositions of the spaces of differential forms on $N$ determine decompositions of 
the exterior differential operator $d$. 
In particular, it is possible to introduce first order differential operators $d^p_q:\Lambda_p\to\Lambda_q$ akin to those introduced by 
Bryant and Harvey on torsion-free $G_2$ manifolds and obtain new formulae that are analogous to those reviewed in 
Proposition \ref{propBryantFormulae}. Here, the relevant spaces are 
\[
\Lambda_1\coloneqq \Lambda^0(N),\quad 
\Lambda_6 \coloneqq \Lambda^1(N),\quad 
\Lambda_8 \coloneqq \Lambda^2_8(N),\quad
\Lambda_{12}\coloneqq \Lambda^3_{12}(N), 
\]
so that $p,q\in\{1,6,8,12\}$.  
We will divide the discussion into various steps. 

\smallskip

First, $d^1_6=d:\Lambda_1\to\Lambda_6$ is the usual differential acting on functions and $d^6_1 = (d^1_6)^* = d^*:\Lambda_6\to\Lambda_1$ 
is the codifferential acting on $1$-forms. 

\smallskip

Let $\alpha\in\Lambda^1(N)$, we define 
\[
\begin{split}
d^6_{6}:\Lambda_6\to\Lambda_{6},&\quad d^6_6\alpha\coloneqq \star d(\alpha\w\Omega^-),\\
d^6_{8}:\Lambda_6\to\Lambda_{8},&\quad d^6_8\alpha\coloneqq \pi^2_8(d\alpha),\\
d^6_{12}:\Lambda_6\to\Lambda_{12},&\quad d^6_{12}\alpha \coloneqq \pi^3_{12}\left(d^*(\alpha\wedge\Omega^-)\right),
\end{split}
\]
We then have the following. 
\begin{lem}\label{l:d66d612}
For every $\alpha\in\Lambda^1(N)$
\begin{enumerate}[1)]
\item $d\alpha = -\frac13 d^6_1(J^*\alpha) \cdot\omega -\frac12 \star\left(J^*d^6_6\alpha\w\Omega^+ \right) + d^6_8\alpha$;
\item $d\star(\alpha\w\Omega^-) =	-\frac12\,d^6_1\alpha\cdot\Omega^+ 
						+ \frac12\,d^6_1(J^*\alpha)\cdot\Omega^- + \frac12 J^* d^6_6\alpha \w\omega + \star d^6_{12}\alpha. $
\end{enumerate}
Moreover, $d^6_1(J^*\alpha) = -\star d \star J^*\alpha = -\frac12 \star d(\alpha\w\omega^2)$. 
\end{lem}
\begin{proof}
\begin{enumerate}[1)]
\item The $2$-form $d\alpha$ decomposes according to the splitting of $\Lambda^2(N)$ as
\[
d\alpha = a\cdot \omega + \star(\alpha'\wedge\Omega^+) + \sigma,
\]
where $a\in \Lambda^0(N)$, $\alpha'\in\Lambda^1(N)$ and $\sigma\in\Lambda^2_8(N)$. 
By definition, $\sigma = \pi^2_8(d\alpha) = d^6_8\alpha$. 
Wedging the expression of $d\alpha$ by $\omega^2$, we obtain
\[
d\alpha\w\omega^2 = a\cdot \omega^3 = 6a \cdot \vol_h. 
\]
On the other hand, by Lemma \ref{l:SU3ids} 
\[
d\alpha\w\omega^2 = d(\alpha\w\omega^2) = 2 d\star(J^*\alpha), 
\]
and thus
\[
a = \frac13 \star d\star(J^*\alpha) = -\frac13 d^*(J^*\alpha). 
\]
Wedging now the expression of $d\alpha$ by $\Omega^-$ and using Lemma \ref{l:SU3ids}, we have
\[
d\alpha\w\Omega^- =  \star(\alpha'\wedge\Omega^+)\w\Omega^- = -2\star(J^*\alpha'). 
\]
Therefore $d^6_6\alpha = \star d(\alpha\w\Omega^-) = 2 J^*\alpha'$, whence
\[
\alpha' = -\frac12 J^*d^6_6\alpha. 
\]
\item The $3$-form $d\star(\alpha\w\Omega^-)$ decomposes according to the decomposition of $\Lambda^3(N)$ as follows
\[
d\star(\alpha\w\Omega^-) =	h^+\cdot\Omega^+ 
						+ h^-\cdot\Omega^- + \alpha_1 \w\omega + \rho,
\]
for $h^\pm\in \Lambda^0(N)$, $\alpha_1\in\Lambda^1(N)$ and $\rho\in\Lambda_{12}$. 
By definition, $\star \rho = \pi^3_{12}(\star d \star(\alpha\w\Omega^-)) = - d^6_{12}\alpha$, thus $\rho = \star d^6_{12}\alpha$. 
Wedging the expression of $d\star(\alpha\w\Omega^-)$ by $\Omega^+$ we obtain
\[
d\star(\alpha\w\Omega^-)\w\Omega^+ = h^-\,\Omega^-\w\Omega^+ = -4\,h^-\vol_h.
\]
Using Lemma \ref{l:SU3ids} we also have
\[
d\star(\alpha\w\Omega^-)\w\Omega^+ = d(\star(\alpha\w\Omega^-)\w\Omega^+) = d(2\star J^*\alpha).
\]
Therefore
\[
h^- = -\frac12 \star d \star (J^*\alpha) = \frac12 d^*(J^*\alpha) = \frac12 d^6_1(J^*\alpha). 
\]
Similarly, wedging the expression of $d\star(\alpha\w\Omega^-)$ by $\Omega^-$ we obtain 
\[
h^+ = \frac14\star \left(d\star(\alpha\w\Omega^-)\w\Omega^-\right) = \frac12 \star d \star \alpha = -\frac12 d^6_1\alpha. 
\]
Finally, wedging the expression of $d\star(\alpha\w\Omega^-)$ by $\omega$ and using Lemma \ref{l:SU3ids} we get
\[
\alpha_1 \w\omega^2 = d\star(\alpha\w\Omega^-)\w\omega = d\left(\star(\alpha\w\Omega^-)\w\omega\right) 
					= d(\alpha\w\Omega^-) = -\star d^6_6\alpha,
\]
and $\alpha_1 \w\omega^2 = 2\star(J^*\alpha_1)$, so that $J^*\alpha_1 = -\frac12 d^6_6\alpha$ and thus
\[
\alpha_1 = \frac12 J^*d^6_6\alpha. 
\]
\end{enumerate}
\end{proof}

\smallskip

Consider now $\beta\in\Lambda_8$. Then, $d^*\beta\in\Lambda_6$ and we  define
\[
d^8_6:\Lambda_8\to\Lambda_6,\quad 
d^8_6\beta \coloneqq d^*\beta = - \star d \star \beta = \star d (\beta\wedge\omega) = \star(d\beta\wedge\omega), 
\]
where we used the identity $\beta\wedge\omega = -\star\beta$. 
We also define $d^{8}_{12}$ as the projection of $d\beta$ onto $\Lambda^3_{12}$:
\[
d^{8}_{12}:\Lambda_8\to\Lambda_{12},\quad d^{8}_{12}\beta \coloneqq \pi^3_{12}(d\beta). 
\]
We then have the following. 
\begin{lem}\label{l:d86d812}	
Let $\beta\in\Lambda_8$, then
\[
\begin{split}
d\beta 		&= \frac12 J^* d^8_6\beta \wedge \omega + d^{8}_{12}\beta,\\
d\star\beta	&= \star d^8_6\beta. 
\end{split}
\]
\end{lem}
\begin{proof}
The differential of $\beta$ decomposes according to the splitting of $\Lambda^3(N)$. 
Since $\beta\wedge\Omega^\pm=0$, we see that $d\beta\wedge\Omega^\pm=0$ and thus $d\beta$ has no components along $\Omega^\pm$. 
Therefore
\[
d\beta = \alpha_1 \wedge \omega + d^{8}_{12}\beta,
\]
for a certain 1-form $\alpha_1$. Using again the identity $\beta\wedge\omega=-\star\beta$, we have
\[
- d\star\beta = d\beta\wedge\omega = \alpha_1 \wedge \omega^2 = 2\star(J^*\alpha_1),
\] 
whence it follows that 
\[
\alpha_1 = \frac12 J^* d^*\beta = \frac12 J^* d^8_6\beta. 
\]
\end{proof}

Let us now consider a $3$-form $\rho\in\Lambda_{12}$. 
The differential operator $d^{12}_6:\Lambda_{12}\to\Lambda_6$ is defined as the formal adjoint of the operator $d^6_{12}$,  
so for every $\alpha\in\Lambda^1_c(N)$
\[
\begin{split}
\int_N (d^{12}_6\rho,\alpha)\vol_h	&=	\int_N (\rho,d^6_{12}\alpha) \vol_h	=	\int_N (\rho, d^\star(\alpha\wedge\Omega^-)) \vol_h \\
							&=	\int_N (d\rho, \alpha\wedge\Omega^-) \vol_h 	
							=	\int_N (-\star(\star d\rho\wedge\Omega^-),\alpha) \vol_h.		
\end{split}
\]
Therefore, 
\begin{equation}\label{d126}
d^{12}_6\rho  =  - \star(\star d\rho\wedge\Omega^-). 
\end{equation}
The 4-forms $d\rho$ and $d\star\rho$ decompose according to the splitting
\[
\Lambda^4(N) = \Lambda^0(N)\cdot\omega^2 \oplus \Lambda^1(N)\wedge\Omega^+ \oplus \star \Lambda^2_8(N). 
\]
Since $\rho\wedge\omega=0$, we have
\begin{equation}\label{e:drho}
\begin{split}
d\rho		&= \alpha_1\wedge\Omega^+ + \star \sigma,	\\
d\star\rho	&= \tilde{\alpha}_1\wedge\Omega^+ + \star\tilde\sigma,
\end{split}
\end{equation}
for certain $\alpha_1,\tilde{\alpha}_1\in\Lambda^1$ and $\sigma,\tilde{\sigma}\in\Lambda^2_8(N)$. From \eqref{d126} we obtain
\[
d^{12}_6\rho 	=  - \star(\star d\rho\wedge\Omega^-) = - \star(\star(\alpha_1\wedge\Omega^+)\wedge\Omega^-) 
=  \star(2\star J^*\alpha_1) = -2J^*\alpha_1,
\]
and thus
\[
d\rho	= \frac12 J^* d^{12}_6\rho\wedge\Omega^+ + \star \sigma.
\]

We now define $d^{12}_8:\Lambda_{12}\to\Lambda_8$ as follows for every $\rho\in\Lambda_{12}$: 
\[
d^{12}_8\rho \coloneqq \pi^2_8(d^* \rho). 
\]
Then, for every $\beta\in\Lambda^2_8(N)$ with compact support we have
\[
\int_N(d^{12}_8\rho,\beta)\vol_h = \int_N(d^*\rho,\beta)\vol_h = \int_N(\rho,d\beta)\vol_h = \int_N(\rho,d^8_{12}\beta)\vol_h,
\]
whence it follows that $d^{12}_8$ is  the formal adjoint  of $d^8_{12}$.

Since $\star\rho\in\Lambda_{12}$, we also have
\[
d^{12}_8(\star\rho) = \pi^2_8(d^* \star\rho) = \pi^2_8(\star d \rho) =\sigma. 
\]
Therefore, \eqref{e:drho} becomes
\[
\begin{split}
d\rho		&= \frac12 J^* d^{12}_6\rho \wedge\Omega^+ + \star d^{12}_8(\star\rho),	\\
d\star\rho	&= \tilde{\alpha}_1\wedge\Omega^+ - \star d^{12}_8\rho. 
\end{split}
\]

The next lemma establishes the link between  $d^{12}_6\rho$ and the 1-form $\tilde{\alpha}_1$ appearing above. 
\begin{lem}
$\tilde{\alpha}_1 = - J^*{{\alpha}_1} = \frac12\, d^{12}_6\rho$.
\end{lem}
\begin{proof}
Let us consider the 7-manifold $M\coloneqq N\times \R$ endowed with the torsion-free G$_2$-structure 
induced by the Calabi-Yau SU(3)-structure on $N$: 
\[
\ph = \omega \wedge dt + \Omega^+,\qquad \psi = \frac12\omega^2 + \Omega^-\wedge dt. 
\]
We will denote the Hodge operator in dimension $k\in\{6,7\}$ by $\star_k$ and the codifferential by $\delta_k$. 

Since $\rho\in\Lambda^3_{12}(N)$, we see that $\rho\wedge\ph=0$ and $\rho\wedge\psi = 0$, 
so that $\rho \in \Lambda^3_{27}(M)$. 
In particular, it satisfies the identity \eqref{DifId327}
\[
\delta_7\rho\wedge\psi = \ph \wedge  \star_7 d\rho. 
\]
Let us compute the LHS and the RHS separately. We have
\[
\delta_7\rho = -\star_7 d \star_7\rho = -\star_7 d (\star_6\rho\wedge dt) = - \star_7(d\star_6\rho\wedge dt) = -\star_6d\star_6\rho = \delta_6\rho,
\]
and thus
\[
\begin{split}
\delta_7\rho\wedge \psi	&=	-\star_6 d \star_6 \rho \wedge \left(\frac12\omega^2 + \Omega^-\wedge dt\right)\\
				&=	-\left(\star_6(\tilde{\alpha}_1\wedge\Omega^+) + \tilde\sigma\right) \wedge  \left(\frac12\omega^2 + \Omega^-\wedge dt\right)\\
				&=	-\star_6(\tilde{\alpha}_1\wedge\Omega^+)\wedge \Omega^-\wedge dt \\
				&=	\left(2\star_6 J^*\tilde{\alpha}_1\right)\wedge dt. 
\end{split}
\]
On the other hand
\[
\star_7 d\rho = \star_6 d\rho \wedge dt = \left(\star_6(\alpha_1\wedge\Omega^+) +  \sigma\right)\wedge dt,
\]
so that
\[
\begin{split}
\ph \wedge  \star_7 d\rho	&=	\left(\omega\wedge dt + \Omega^+\right) \wedge  \left(\star_6(\alpha_1\wedge\Omega^+) +  \sigma\right)\wedge dt\\
					&=	\star_6(\alpha_1\wedge\Omega^+) \wedge \Omega^+ \wedge dt\\
					&=	2\star_6\alpha_1 \wedge dt. 
\end{split}
\]
Therefore, $\alpha_1 = J^*\tilde{\alpha}_1$, so $\tilde{\alpha}_1 = -J^*\alpha_1$. 
\end{proof}

\begin{cor}\label{cor:SU3idJ}
    For every $\rho\in\Lambda_{12}$, we have $J^*d^{12}_6\rho = - d^{12}_6 \star\rho$. 
    Consequently, for every $1$-form $\alpha\in\Lambda^1$, $d^6_{12}J^*\alpha = - \star d^6_{12}\alpha.$
\end{cor}
\begin{proof}
    We first observe that, from \eqref{d126}, we have 
    \[
    d^{12}_6\star\rho =- \star(\star d \star \rho\wedge\Omega^-) = -\star(\star(\tilde{\alpha}_1\w\Omega^+)\w\Omega^-) = -2J^*\tilde{\alpha}_1. 
    \]
    From the proof of the previous lemma, we then obtain 
    \[
    \frac12 J^* d^{12}_6\rho = \alpha_1 = J^*\tilde{\alpha}_1 = -\frac12 d^{12}_6\star\rho,
    \]
    whence the first identity follows. Now, for all compactly supported differential forms $\alpha\in\Lambda^1$ and $\rho\in\Lambda_{12}$, we have
    \[
    \int_N (d^{12}_6\star\rho, \alpha) \vol_h = \int_N (\star\rho, d^6_{12}\alpha) \vol_h = \int_N (\rho, -\star d^6_{12}\alpha) \vol_h,
    \]
    and
    \[
    \int_N (d^{12}_6\star\rho, \alpha) \vol_h = \int_N (-J^*d^{12}_6\rho, \alpha) \vol_h = \int_N (d^{12}_6\rho, J^*\alpha) \vol_h 
                                                = \int_N (\rho, d^6_{12}J^*\alpha) \vol_h,
    \]
    since both the operator $\star$ on 3-forms and the operator $J^*$ on 1-forms are skew-adjoint. This proves the second identity. 
\end{proof}

Summing up, we have the following.
\begin{lem}\label{l:d312}
Let $\rho\in\Lambda_{12}$, then
\[
\begin{split}
d\rho		&= \frac12 J^* d^{12}_6\rho \wedge\Omega^+ + \star d^{12}_8(\star\rho),	\\
d\star\rho	&= \frac12 d^{12}_6\rho\wedge\Omega^+ - \star d^{12}_8\rho. 
\end{split}
\]
\end{lem}
 
\begin{remark}
Recall that every $\rho\in\Lambda^3_{12}(N)$ satisfies the identity $\star\rho=J^*\rho$. 
\end{remark}

As in the $G_2$ case, the condition $d^2=0$ is equivalent to second order identities on the operators $d^p_q$. We will need the next one. 
\begin{lem}\label{l:d2su3}
The following second order identity holds
\[
d^6_8 J^* d^6_6 + 2\, d^{12}_8 d^6_{12} =0.
\]
\end{lem}
\begin{proof}
Using Lemma \ref{l:d66d612}, we see that 
\[
\begin{split}
 0 	&= d^2\star(\alpha\w\Omega^-) \\
 	&=	-\frac12\,dd^6_1\alpha\w\Omega^+ 
						+ \frac12\,dd^6_1(J^*\alpha)\w\Omega^- + \frac12 d(J^* d^6_6\alpha) \w\omega + d\star d^6_{12}\alpha. 
\end{split}
\]
Let us focus on the component in $\Lambda^4_8(N)$
\[
0 = \pi^4_8\left(d^2\star(\alpha\w\Omega^-)\right) = \pi^4_8\left( \frac12 d(J^* d^6_6\alpha) \w\omega + d\star d^6_{12}\alpha\right).
\]
Using 1) of Lemma \ref{l:d66d612} and Lemma \ref{l:SU3ids}, we obtain
\[
\begin{split}
d(J^* d^6_6\alpha) \w\omega &= \left(-\frac13 d^6_1(J^*J^* d^6_6\alpha) \cdot\omega -\frac12 \star\left(J^*d^6_6J^* d^6_6\alpha\w\Omega^+ \right) + d^6_8J^* d^6_6\alpha\right)\w\omega\\
&= \frac13 d^6_1( d^6_6\alpha) \cdot\omega^2 -\frac12 \star\left(d^6_6J^* d^6_6\alpha\w\Omega^- \right)\w\omega + d^6_8J^* d^6_6\alpha\w\omega\\
&= \frac13 d^6_1( d^6_6\alpha) \cdot\omega^2 -\frac12 (d^6_6J^* d^6_6\alpha)\w\Omega^-  -\star d^6_8J^* d^6_6\alpha,
\end{split}
\]
so that
\[
\pi^4_8\left( \frac12 d(J^* d^6_6\alpha) \w\omega\right) = -\frac12 \star d^6_8J^* d^6_6\alpha. 
\]
Using now Lemma \ref{l:d312}, we get
\[ 
d\star d^6_{12}\alpha = \frac12 d^{12}_6d^6_{12}\alpha\wedge\Omega^+ - \star d^{12}_8d^6_{12}\alpha,
\]
and thus
\[
\pi^4_8\left( d\star d^6_{12}\alpha\right) = - \star d^{12}_8d^6_{12}\alpha. 
\]
Therefore, the component of $d^2\star(\alpha\w\Omega^-)$ in $\Lambda^4_8(N)$ is
\[
\pi^4_8\left(d^2\star(\alpha\w\Omega^-)\right)  = -\star\left(\frac12 d^6_8J^* d^6_6\alpha + d^{12}_8d^6_{12}\alpha\right),
\]
and the thesis follows. 
\end{proof}

\subsection{The global $dd^{\Omega^+}$ lemma}

We now study the characterization of the space $\Im(dd^{\Omega^+})$. 
We begin with the following analogue of Lemma \ref{l:ddphi_features}. 

\begin{lem} \label{l:ddOmega_features}
Let $(N,h,J,\omega,\Omega)$ be a Calabi-Yau $6$-manifold. 
Given $f\in\Lambda^0(N)$, the $3$-form $dd^{\Omega^+} f$ has the following properties:
\begin{enumerate}
\item $dd^{\Omega^+} f$ is exact;
\item $dd^{\Omega^+} f = \lambda_{\Omega^+}(\nabla\nabla f)\in \Lambda^0(N)\cdot \Omega^+\oplus\Lambda^3_{12}(N)$ and 
\[
dd^{\Omega^+} f=-\frac12\Delta f\cdot\Omega^+ + \star d^6_{12}(df);
\]
\item $\pi^2_{8}d^*(dd^{\Omega^+} f)=0$.
\end{enumerate}
\end{lem} 
\begin{proof} 

1.~is obvious. 
As for 2., we first observe that
\[
dd^{\Omega^+} f = d (\nabla f \lrcorner \Omega^+) = d\star(df\w\Omega^-). 
\]
We can then apply Lemma \ref{l:d66d612} with $\alpha=df$, obtaining 
\[
dd^{\Omega^+} f = -\frac12\Delta f\cdot\Omega^+ + \star d^6_{12}(df) ~\in~  \Lambda^0(N)\cdot \Omega^+\oplus\Lambda^3_{12}(N),
\]
since 
\[
d^6_6(df) = \star d(df \w\Omega^-) = 0,
\]
and
\[
d^6_1(J^*df) =  - \frac12 \star d(df\w\omega^2) = 0. 
\]

Let us now prove 3. 
First, for any $u\in \Lambda^0(N)$, the $2$-form
\[
d^*(u\cdot\Omega^+) = - \star d \star (u\cdot\Omega^+) =  - \star d  (u\cdot \Omega^-) = - \star  (du\w \Omega^-)
\] 
has no component in $\Lambda^2_{8}(N)$. 
Furthermore, using Lemma \ref{l:d2su3} and Corollary \ref{cor:SU3idJ}, we obtain
\[
\pi^2_{8}d^*(\star d^6_{12}(df)) = \pi^2_8 \star dd^6_{12}(df) = d^{12}_8\star d^6_{12}(df) = - d^{12}_8d^6_{12}J^*(df) = \frac12 d^6_8 J^* d^6_6 J^*(df),
\]
and we conclude observing that
\[
d^6_6 J^*(df) = \star d(J^*(df)\w\Omega^-) = -\star d(df\w\Omega^+) = 0. 
\]
\end{proof}

This result motivates the following. 

\begin{definition}
Let $(N,h,J,\omega,\Omega)$ be a Calabi-Yau $6$-manifold. Set $\Lambda^3_1(N) \coloneqq \Lambda^0(N)\cdot \Omega^+$ and
\[
C\coloneqq \left(\Lambda^3_1(N)\oplus\Lambda^3_{12}(N)\right)\cap \Ker(\pi^2_{8}d^*), \qquad K \coloneqq C\cap \Ker(d).
\]
\end{definition}

Lemma \ref{l:ddOmega_features} shows that $\Im(dd^{\Omega^+})\subseteq K$. 
In order to characterize its complement, we will need the next result. 

\begin{lem}\label{l:princsymbSU3}
The principal symbol $\sigma(dd^{\Omega^+})$ 
of the linear differential operator $dd^{\Omega^+}:\Lambda^0(N)\rightarrow\Lambda^3_1(N)\oplus\Lambda^3_{12}(N)$ gives the homomorphisms
\[ 
\sigma(dd^{\Omega^+})|_{x}(\xi):\R\to\Lambda^3_1(T^*_xN)\oplus\Lambda^3_{12}(T^*_xN) ,\quad 
\sigma(dd^{\Omega^+})|_{x}(\xi) : c \mapsto c\,\xi \wedge (\xi^\sharp\lrcorner \Omega^+),
\]
for all $x\in N$ and $\xi\in T^*_xN$. 
This symbol is injective.
\end{lem}
\begin{proof}
Let $(e^1,\ldots,e^6)$ be a special unitary coframe of $T^*_xN$, so that 
\[
\Omega^+|_x =  e^{135} - e^{146} - e^{236} - e^{245}.
\]
Let $\xi \in T^*_xN \smallsetminus\{0\}$. It is not restrictive assuming that $\xi = a e^1,$ for a certain $a\neq0$, 
as the group SU(3) acts transitively on the unit sphere in $T_x^*N$ preserving $\Omega^+|_x$.  
We then have
\[
\xi \wedge (\xi^\sharp\lrcorner \Omega^+) = a^2 \left(e^{135} - e^{146}\right). 
\]
From this, it easily follows that the map $\sigma(dd^{\Omega^+})|_{x}(\xi)$ is injective. 
\end{proof}

We are now ready to prove the Calabi-Yau analogue of Theorem \ref{ThmSplittingKahler} (K\"ahler case) and Theorem  \ref{thm:G2Hodge} ($G_2$ case). 
Again, the technical details are similar to those seen above, so we will provide only a sketch of the main arguments.

\begin{thm}\label{thm:SU3Hodge}
Let $(N,h,J,\omega,\Omega)$ be a compact Calabi-Yau $6$-manifold.  Then there exists an $L^2$-orthogonal decomposition 
\[
K=\Im(dd^{\Omega^+})\oplus\mathcal{H}^3_{1}\oplus\mathcal{H}^3_{12},
\]
where $\mathcal{H}^3_{1} = \{u\cdot\Omega^+\in\Lambda^3_1(N)\,:\, \Delta(u\cdot\Omega^+)=0\} = \R\cdot\Omega^+$ and
$\mathcal{H}^3_{12} = \{\rho\in\Lambda^3_{12}(N)\,:\, \Delta\rho=0 \}$. 
\end{thm}

\begin{proof}
Consider the operator
\[
dd^{\Omega^+}:\Lambda^0(N)\rightarrow \Lambda^3_1(N)\oplus \Lambda^3_{12}(N), 
\]
endow both spaces with the $L^2$ metric, and consider a generic element $u\cdot\Omega^+ + \rho \in \Lambda^3_1(N)\oplus \Lambda^3_{12}(N)$. 
Then 
\begin{align*}
\int_N (dd^{\Omega^+}f, u\cdot\Omega^+ + \rho )\vol_h
&=\int_N (\star(df\w\Omega^-),d^*(u\cdot\Omega^+ + \rho))\vol_h\\
&=\int_N df\wedge\Omega^- \wedge d^*(u\cdot\Omega^+ + \rho)\\
&=-\int_N f\,  dd^*(u\cdot\Omega^+ + \rho)\w \Omega^-\\
&=-\int_N (f,\star(dd^*(u\cdot\Omega^+ + \rho)\wedge\Omega^-))\vol_h,
\end{align*}
Therefore, within the given spaces, the adjoint of $dd^{\Omega^+}$ is the map
\[
(dd^{\Omega^+})^t = - \star(dd^*(\cdot)\wedge\Omega^-) : \Lambda^3_1(N)\oplus \Lambda^3_{12}(N) \to \Lambda^0(N).
\]

Consider now the unbounded operator $dd^{\Omega^+}:L^2(N)\to L^2(\Lambda^3_1(N)\oplus\Lambda^3_{12}(N))$ with domain $D=H^2(N)$. 
Its adjoint $(dd^{\Omega^+})^*$ is an appropriate extension of $(dd^{\Omega^+})^t$. 
Using Proposition \ref{p:abstractdecomp} and Lemma \ref{l:princsymbSU3} we get the decomposition
\[
L^2(\Lambda^3_1(N)\oplus\Lambda^3_{12}(N))=\Im(dd^{\Omega^+})\oplus \Ker((dd^{\Omega^+})^*). 
\]

In order to refine this decomposition, we consider the operators 
\[
d:\Lambda^3(N)\rightarrow\Lambda^4(N), \quad \pi^2_{8}d^*:\Lambda^3(N)\rightarrow\Lambda^2_{8}(N),
\]
and the auxiliary operator $Q:\Lambda^3_1(N)\oplus\Lambda^3_{12}(N)\rightarrow\Lambda^0(N)\oplus\Lambda^4(N)\oplus\Lambda^2_{8}(N)$
\[
Q:\alpha\mapsto\left((dd^{\Omega^+})^t(\alpha),d\alpha,\pi^2_{8}d^*(\alpha)\right).
\]
Its principal symbol $\sigma(Q)$ gives the following homomorphism at $x\in N$ 
\[
\sigma(Q)|_x(\xi): \alpha \mapsto \left( \star (\xi \wedge(\xi^\sharp\lrcorner\alpha)\wedge\Omega^-), \xi\wedge\alpha, 
-\pi^2_{8}(\xi^\sharp \lrcorner \alpha)\right),
\]
for every $\xi\in T^*_xN$. We now show that the principal symbol is injective. 
Let $\xi\neq0$ and consider $\alpha\in \Ker(\sigma(Q)|_x(\xi))\subseteq \Lambda^3_1(T^*_xN)\oplus\Lambda^3_{12}(T^*_xN)$. 
Then $\alpha\w\omega=0$, whence it follows that 
$\xi^\sharp \lrcorner \alpha\w\omega =  \alpha\w \xi^\sharp \lrcorner\omega$, thus 
$\xi^\sharp\lrcorner \alpha \wedge \omega^2 =0$. This shows that $\xi^\sharp\lrcorner \alpha \in\Lambda^2_6(T^*_xN)\oplus \Lambda^2_8(T^*_xN)$. 
The condition $\xi\wedge\alpha=0$ implies
\[
0 =  \star (\xi \wedge(\xi^\sharp\lrcorner\alpha)\wedge\Omega^-) = |\xi|^2 \star(\alpha\wedge\Omega^-),
\]
thus $\alpha\in\Lambda^3_{12}(T^*_xN)$. 
Now, consider a special unitary coframe $(e^1,\ldots,e^6)$ of $T^*_xN$  
and use the SU(3) action to reduce to the case $\xi = a e^1$ with $a\neq0$.  
Writing $\alpha\in\Lambda^3_{12}$ in terms of its 12 independent coefficients, the condition $0=\xi\w\alpha=a\,e^1\wedge\alpha=0$ gives ten linear equations, 
corresponding to the coefficients of $e^{1ijk}$, where $2\leq i<j<k\leq6$. Substituting their solutions shows that
$\xi^\sharp\lrcorner \alpha \wedge \Omega^-=0$, 
whence it follows that $\xi^\sharp\lrcorner \alpha \in \Lambda^2_{8}(T^*_xN)$. 
On the other hand, since $\alpha\in \Ker(\sigma(Q)|_x(\xi))$ we also have $\pi^2_{8}(\xi^\sharp \lrcorner \alpha)=0$, 
which forces $\xi^\sharp\lrcorner \alpha = 0$.  
This in turn implies  $|\xi|^2\alpha = \xi \wedge (\xi^\sharp \lrcorner \alpha) =0$,
whence $\alpha=0$.

Since the principal symbol is injective, we get 
\[
\Ker\left((dd^{\Omega^+})^*|_{\overline{\Ker(d)\cap\Ker(\pi^2_{8}d^*)}}\right) = 
\Ker\left((dd^{\Omega^+})^t|_{\Ker(d)\cap\Ker(\pi^2_{8}d^*)}\right)=\Ker((dd^{\Omega^+})^t|_{K}).
\] 
We now focus on $\Ker((dd^{\Omega^+})^t|_{K})$. Since
\[
(dd^{\Omega^+})^t|_{K}=-\star(\Delta|_{K}(\cdot)\wedge\Omega^-), 
\]
we have
\[
\Ker((dd^{\Omega^+})^t|_{K})=\Ker(\Delta(\cdot)\wedge\Omega^-)\cap K,
\]
where $\Delta$ has domain $\Lambda^3_1(N)\oplus\Lambda^3_{12}(N)$.
Recall that $\Delta$ preserves types and that $\rho\wedge\Omega^-=0$, for all $\rho\in\Lambda^3_{12}(N)$. Consequently, we have
\[
\Delta(u\cdot\Omega^+ + \rho)\w\Omega^- = \Delta(u\cdot\Omega^+)\w\Omega^- + \Delta\rho\w\Omega^-  = 4\Delta u \star1. 
\]
It follows that 
\[
\Ker(\Delta(\cdot)\wedge\Omega^-) \cap K =(\R\cdot\Omega^+ \oplus\Lambda^3_{12}(N))\cap K 
=\R\cdot\Omega^+\oplus(\Lambda^3_{12}(N)\cap K), 
\]
since $\R\cdot \Omega^+ \leq K$.

By Lemma \ref{l:d312}
\[
d\rho=0 \quad\Leftrightarrow\quad d^{12}_6\rho=0\quad\mbox{and}\quad d^{12}_{8}\star\rho=0.
\]
Comparing with the formula for $\d^*\rho=0$, we see that the additional condition $\pi^2_{8}(d^*\rho)=0$ suffices to obtain 
$\rho\in\mathcal{H}^3_{12}$. Therefore $\Lambda^3_{12}(N)\cap K = \mathcal{H}^3_{12}$. 

Summing up, we have 
$\Ker((dd^{\Omega^+})^t|_{K})=\R\cdot\Omega^+ \oplus\mathcal{H}^3_{12}=\mathcal{H}^3_{1}\oplus \mathcal{H}^3_{12}$, and 
\[
L^2(\Lambda^3_1(N)\oplus\Lambda^3_{12}(N))\cap\overline{\Ker(d)\cap\Ker(\pi^2_{8}d^*)} = 
\Im(dd^{\Omega^+})\oplus\mathcal{H}^3_{1}\oplus \mathcal{H}^3_{12}.
\]
Intersecting both sides with $K$, the thesis follows. 

\end{proof}

\begin{cor}[Global $dd^{\Omega^+}$ lemma]\label{cor:SU3globallemma} 
Let $(N,h,J,\omega,\Omega)$ be a compact Calabi-Yau $6$-manifold.
Fix a form $\alpha\in\Lambda^3_{1}(N)\oplus\Lambda^3_{12}(N)$. 
If (i) $\alpha$ is globally $d$-exact and (ii) $\pi^2_{8}(d^*\alpha)=0$, then $\alpha$ is globally $dd^{\Omega^+}$-exact.
\end{cor}

\subsection{Cohomology}

Also in this case we find a cohomological point of view on some of the spaces of harmonic 3-forms on $N.$ 
Indeed, we can introduce the complex
\begin{equation}\label{eq:complexSU3}
\Lambda^0(N)	\xrightarrow[]{dd^{\Omega^+}} \Lambda^0(N)\cdot \Omega^+\oplus\Lambda^3_{12}(N)  
			\xrightarrow[]{(d,\pi^2_{8}d^*)} \Lambda^4(N)\oplus \Lambda^2_8(N),
\end{equation}
and define the $dd^{\Omega^+}$-cohomology space 
\[
H^{\Omega^+}(N) \coloneqq (\ker(d)\cap\ker(\pi^2_8d^*))/\Im(dd^{\Omega^+}) = K/\Im(dd^{\Omega^+}). 
\]
Theorem \ref{thm:SU3Hodge} can then be summarized as follows. 

\begin{cor}\label{cor:SU3coh}
The linear map
\[
H^{\Omega^+}(N)\to H^3(N;\R),\quad [\alpha]_{\Omega^+} \mapsto [\alpha]
\]
is injective and defines an isomorphism 
$H^{\Omega^+}(N)\simeq \mathcal{H}^3_{1} \oplus \mathcal{H}^3_{12}.$
\end{cor}

Once again, one could alternatively prove Theorem  \ref{thm:SU3Hodge} using ellipticity at the middle term of \eqref{eq:complexSU3}, 
as discussed in remarks \ref{rem:alternativeKahler} and  \ref{rem:alternativeG2}.

\begin{remark}
The parallel calibration $\Omega^-$ and the corresponding operator $dd^{\Omega^-}$ can be used to introduce the $dd^{\Omega^-}$-cohomology space,     
which turns out to be isomorphic to $\R\cdot\Omega^-\oplus  \mathcal{H}^3_{12}$, thus to the $dd^{\Omega^+}$-cohomology space. 
The discussion is similar to the one involving the operator $dd^{\Omega^+}$, so we omit the details. 
\end{remark}

\section{A geometric viewpoint  on $dd^\phi$-cohomology}\label{s:gerbes}
The goal of this section is to provide a geometric interpretation of certain classes in the $dd^\phi$-cohomology space $H^\phi(M)$. 
This is inspired by the original motivation in \cite{BottChern} for defining the cohomology space $H^{1,1}_{BC}(M)$, related to the theory of holomorphic line bundles. We shall focus on the $G_2$ case using the analogous theory of gerbes, as presented in \cite{HitchinGerbes}.

\begin{remark}
Below, we shall restrict our attention to compact oriented manifolds simply because this allows us to rely on Poincar\'e duality, thus obtaining a more complete theory.
\end{remark}

\subsection{Gerbes vs. line bundles}

Recall the following notion.

\begin{definition}
Let $M$ be a smooth compact oriented $n$-dimensional manifold. 
A {\em gerbe} on $M$ is a \v{C}ech cohomology class $\mathcal{G}\in H^2(M;C^{\infty}(\C^*))$.
\end{definition}

Let us emphasize the following facts, referring to \cite{HitchinGerbes} for details:
\begin{enumerate}[(1)]
\item Recall that $H^0(M;C^{\infty}(\C^*))$ parametrizes global smooth functions on $M$ with values in $\C^*$, while $H^1(M;C^{\infty}(\C^*))$ parametrizes isomorphism classes of smooth complex line bundles on $M$. 
Gerbes are the natural next object in this list. 

\item It is sometimes important to distinguish classes and specific representatives. By definition, given an open covering $\{U_\alpha\}$ of $M$ and up to direct limits,
\begin{itemize} 
\item a specific line bundle is determined by smooth functions $g_{\alpha\beta}:U_\alpha\cap U_\beta\rightarrow\C^*$, satisfying the cocycle condition;
\item to represent a gerbe we need smooth functions $g_{\alpha\beta\gamma}:U_\alpha\cap U_\beta\cap U_\gamma\rightarrow \C^*$, satisfying the cocyle condition.
\end{itemize}
A gerbe can alternatively be represented by a collection of local smooth complex line bundles $L_{\alpha\beta}$ on $U_\alpha\cap U_\beta$, satisfying certain conditions.
\item Both $H^1(M;C^{\infty}(\C^*))$ and $H^2(M;C^{\infty}(\C^*))$ have a natural group structure. We shall refer to this as the tensor product, or as the \textit{twisting operation}. This defines the notion of trivial class. For example, the trivial gerbe $\mathcal{G}_0$ can be represented by the data $g_{\alpha\beta\gamma}\equiv 1$.

Adopting the alternative viewpoint in (2), the tensor product of two gerbes is represented by the tensor product of the corresponding local line bundles. 
\end{enumerate}
Recall that the short exact sequence of groups
$$0\rightarrow\Z\rightarrow\C\rightarrow\C^*\rightarrow 1$$
leads to a short exact sequence of sheaves, thus to a long exact sequence of cohomology groups
\[
\cdots \rightarrow H^k(M;C^{\infty}(\C))\rightarrow H^k(M;C^{\infty}(\C^*))\rightarrow H^{k+1}(M;\Z)\rightarrow H^{k+1}(M;C^{\infty}(\C))\rightarrow \cdots
\]
However, the sheaves $C^{\infty}(\C)$ are examples of fine sheaves. Their cohomology vanishes for $k\geq 0$, so we obtain isomorphisms
$$c_1:H^k(M;C^{\infty}(\C^*))\simeq H^{k+1}(M;\Z).$$
In particular: a class of smooth line bundles or a gerbe is completely determined by its image under $c_1$, known as the \textit{first Chern class}.

Recall also that Poincar\'e duality defines an isomorphism $H^k(M;\Z)\simeq H_{n-k}(M;\Z)$. Composing these isomorphisms, a class of smooth line bundles or a gerbe is  equivalent to a homology class in the dual dimension. This statement concerns classes. It is an interesting question whether it admits a geometric counterpart, and this is one place where the two cases of line bundles and gerbes differ:
\begin{itemize}
\item 
Let $L$ denote a specific representative of a class in $H^1(M;C^{\infty}(\C^*))$. Choose a smooth section transverse to the zero section. Then its zero set defines a submanifold in $M$ whose homology class in $H_{n-2}(M;\Z)$ is dual to $c_1(L)\in H^2(M;\Z)$.
 
More abstractly: in codimension 2, any homology class can be represented by a smooth embedded submanifold.
\item The abstract statement is false in codimension 3. We must thus distinguish the subset (or the subgroup it generates) whose elements are defined by the property that their corresponding homology class in $H_{n-3}(M;\Z)$ contains a smooth embedded submanifold. We shall refer to these elements as \textit{geometric gerbes. }
\end{itemize}

\subsubsection{Connections}
Recall that any line bundle $L$ admits connections, and any connection $\nabla$ defines a curvature 2-form $R$ on $M$. This 2-form is closed, so (up to factors) it defines a cohomology class $[R]\in H^2(M;\R)$. Chern-Weil theory proves that this class is integral and coincides (up to torsion) with $c_1(L)$. 

Let $L$, $L'$ denote two line bundles. Let $\nabla$, $R$ and $\nabla'$, $R'$ denote connections and curvature 2-forms on $L$, $L'$. There is an induced connection on $L\otimes L'$, whose curvature is $R+R'$. From the Chern-Weil viewpoint, this corresponds to the fact that $c_1$ is a group homomorphism. 

In particular, even if $L'$ is the trivial bundle so that the twisting operation is topologically trivial on $L$, 
if we endow $L'$ with a non-flat connection then we will see the effect of twisting on the curvature of $L$.

There exists an analogous notion of connections and curvature on gerbes. Referring to \cite{HitchinGerbes} for details, we shall just mention that if a gerbe is represented by data $\{U_\alpha,g_{\alpha\beta\gamma}\}$, then a connection corresponds to 2-forms $F_\alpha$ on $U_\alpha$ such that, on $U_\alpha\cap U_\beta$, $F_\beta-F_\alpha$ is exact. It follows that $d(F_\beta)-d(F_\alpha)=0$, so the data $\{d(F_\alpha)\}$ defines a closed global 3-form $G$ on $M$: this is the curvature of the connection. As above, its cohomology class $[G]\in H^3(M;\R)$ is integral and coincides (up to torsion) with $c_1$ of the gerbe.
If we represent the gerbe via a collection of local line bundles, a connection on the gerbe can be viewed as a collection of connections on the line bundles. This shows that the effect on curvature of the twisting operation is, as before, to produce the sum of the two curvatures.

\begin{example}\label{e:gerbeconnection}
Let $M$ be a smooth compact oriented manifold. Let $\mathcal{G}_0$ be the trivial gerbe. Choose any global 2-form $F$ on $M$. Given an open covering of $M$, the restrictions $F_\alpha=F_{|U_\alpha}$ define a connection on $\mathcal{G}_0$ since $F_\beta-F_\alpha=0$ is clearly exact. The curvature is $G=dF$. As expected, $[G]=0\in H^3(M;\Z)$.
\end{example}

\subsection{A geometric source of Bott-Chern classes}
Let us now assume that $M$ is a complex manifold and $L$ is a holomorphic line bundle over $M$. In this case there exists a distinguished class of connections: Chern connections, defined in terms of a choice of Hermitian metric $h$ on $L$. 

From our viewpoint, Chern connections have a key property: the curvature $R$ defines a closed real 2-form $\frac{i}{2\pi}R$ of type $(1,1)$. In particular, this form is both $\partial$- and $\bar\partial$-closed. As predicted by the local $\partial\bar\partial$ lemma, $R$ admits local potential functions. These can be found using the definition of Chern connection. Specifically, given any nowhere vanishing local holomorphic section $s$, set $H\coloneqq h(s,s)$. 
Then, locally, $R=\partial\bar\partial(-\log H)$.

Notice that any two metrics $h,h'$ on $L$ are conformal, so they are related by a smooth function $f:M\rightarrow\R$: $h'=e^{-f}h$. We are particularly interested in the following, well-known, fact. 
\begin{prop}\label{prop:BottChern}
Let $L$ be a holomorphic line bundle and $h$ be a Hermitian metric on $L$. Let $R$ be the curvature of its Chern connection. Let $h'=e^{-f}h$ be a second Hermitian metric on $L$. Then the curvature of its Chern connection satisfies $\frac{i}{2\pi}R'=\frac{i}{2\pi}R+\frac{i}{2\pi}\partial\bar\partial f$.
\end{prop}
\begin{proof}
It suffices to substitute $h$ with $h'$ in the formula for the local potential, given above.
\end{proof}
We can summarize this situation as follows \cite{BottChern}:
\begin{itemize}
\item $c_1(L)$ can be interpreted as a (real) Bott-Chern class.
\item Different elements $\frac{i}{2\pi}R$ within this class correspond to the curvatures obtained via different Hermitian metrics on $L$.
\end{itemize}
Changing metric thus provides a perfect geometric counterpart to the analytic process of changing representative in $c_1(L)$, viewed as a Bott-Chern class.

\begin{remark}\label{rem:BottChern}
\cite{BottChern} refers to $c_1(L)$, viewed as an element of $H^{1,1}_{BC}(M)$, as a ``refined Chern class". This language emphasizes the restriction to perturbations of the form $i\partial\bar\partial f$, rather than all exact perturbations.
\end{remark}

We would like to extend this construction to $G_2$ manifolds. In light of Section \ref{s:cddlocal}, it is necessary to avoid local potentials. Let us thus consider the following special case of the above theory.

\begin{example}\label{e:holtrivial}
Let $L_0:=M\times\C$ denote the trivial holomorphic line bundle. 
Let $s$ denote a global nowhere vanishing holomorphic section. Choose any metric $h$. Then the Chern connection has  curvature $\partial\bar\partial(-\log H)$, where $H\coloneqq h(s,s)$. It thus admits a global potential.
\end{example}

Coupling this with the twisting operation provides an alternative to Proposition \ref{prop:BottChern}.

\begin{prop}\label{prop:BottChern_revised}
Let $L$ be a holomorphic line bundle and $h$ be a Hermitian metric on it. Let $R$ be the curvature of its Chern connection. 
Let $L_0\coloneqq M\times\C$ denote the trivial holomorphic line bundle, endowed with the constant metric $h_0$. 
Choose any $f:M\rightarrow\R$. If we twist $(L,h)$ with $(L_0,e^{-f}h_0)$, the curvature of the corresponding Chern connection on $L=L\otimes L_0$ 
satisfies $\frac{i}{2\pi}R'=\frac{i}{2\pi}R+\frac{i}{2\pi}\partial\bar\partial f$.
\end{prop}
\begin{proof}
Let $s$ be a global non-vanishing holomorphic section of $L_0$ such that $h_0(s,s)=1$. 
As in Example \ref{e:holtrivial}, the corresponding Chern connection has curvature $0$. 
If we substitute $h_0$ with $e^{-f}h_0$ and use the properties of the twisting operation, 
we obtain the desired result.
\end{proof}

\begin{remark}
Notice that $c_1(L)$ is an integral class, {\em i.e.}, it belongs to $H^2(M;\Z)$. This implies that changing $L$ spans only a subgroup of Bott-Chern cohomology.

Line bundles admitting holomorphic sections typically generate an even smaller subgroup. In this case, the Poincar\'e dual of $c_1(L)$ can be represented by the codimension 1 complex submanifold (actually, divisor) determined by the zero set of such a section.
\end{remark}

\subsubsection{The role of the global $\partial\bar\partial$ lemma.}
As discussed, on a general complex manifold the relationship between Bott-Chern and de Rham cohomology is unclear. However, if we assume that $M$ is a compact K\"ahler manifold, the global $\partial\bar\partial$ lemma implies that the map from Bott-Chern to de Rham cohomology is injective. We can reformulate this as follows. 

Any given closed $(1,1)$ form defines a Bott-Chern class. In general, it is non-obvious exactly which other closed $(1,1)$ forms belong to that same class. The global $\partial\bar\partial$ lemma allows us to address this issue: any other closed $(1,1)$ form in the same de Rham class belongs to the same Bott-Chern class. We can use this in the following way. 

As already mentioned, curvatures of Chern connections are of type $(1,1)$, closed, and in the de Rham class $c_1(L)$. The global $\partial\bar\partial$ lemma provides the converse statement: any such form $R'$ is the curvature with respect to some metric. 
Indeed, choose a generic metric $h$ on $L$. Let $R$ be its curvature. The $\partial\bar\partial$ lemma implies that $\frac{i}{2\pi}R'=\frac{i}{2\pi}R+\frac{i}{2\pi}\partial\bar\partial f$, for some $f$. According to Proposition \ref{prop:BottChern}, the metric $h' \coloneqq e^{-f}h$ then has curvature $R'$.

Together with Proposition \ref{prop:BottChern_revised} and K\"ahler Hodge theory, this implies the following.

\begin{cor}\label{cor:BottChern}
Let $M$ be a compact K\"ahler manifold and $L$ be a holomorphic line bundle.
\begin{enumerate}
\item There exists a unique harmonic $(1,1)$-form $\sigma\in {\mathcal{H}^{1,1}_\R}$ in the class $c_1(L)$. 
This form is the curvature of the Chern connection on $L$, for an appropriate metric $h$ on $L$.
\item Any other element in the Bott-Chern class $c_1(L)$ has the form $\sigma+\frac{i}{2\pi}\partial\bar\partial f$, for some $f$. 
These elements coincide with the curvatures obtained by twisting $L$, endowed with the fixed metric $h$, 
with the trivial holomorphic line bundle $L_0$ endowed with different Hermitian metrics.
\end{enumerate}
\end{cor}

\begin{remark}
If we change the K\"ahler metric on $M$ we will obtain a different harmonic form $\sigma'\in c_1(L)$. Since it is in the same de Rham class as $\sigma$, the global $\partial\bar\partial$ lemma implies that $\sigma'=\sigma+\frac{i}{2\pi}\partial\bar\partial f'$, for some $f'$. The two forms thus belong to the same Bott-Chern class, as expected.
\end{remark}

\subsection{A geometric viewpoint on classes in $H^\varphi(M)$} Let $(M,\varphi)$ be a compact torsion-free $G_2$ manifold. Let $H^\varphi(M)$ denote its $dd^\varphi$-cohomology space. As already mentioned, the theory of calibrated geometry shows there are  some features in common with K\"ahler manifolds. Here, we are interested in the dual calibration $\psi=\star\varphi$ and in the corresponding class of 4-dimensional \textit{coassociative submanifolds} $C\subseteq M$, defined by the condition $\psi_{|C}=\vol_C$. We shall think of them as analogues of the calibration $\star\omega$ and of the class of complex divisors $D$, characterized by the condition $\star\omega_{|D}=\vol_D$. It is then natural to think of the geometric gerbe $\mathcal{G}$ defined by $C$ as the analogue of the holomorphic line bundle $L$ defined by $D$.

Although we have no analogue for gerbes of the class of Chern connections, there is an obvious refinement of Example \ref{e:gerbeconnection} which has the same flavour as Example \ref{e:holtrivial}.

\begin{example}\label{e:G2gerbeconnection}
Let $(M,\varphi)$ be a compact torsion-free $G_2$ manifold. Let $\mathcal{G}_0$ be the trivial gerbe. Choose any global function $f$ on $M$. Then $F\coloneqq d^\varphi f$ defines a connection on $\mathcal{G}_0$, with curvature $dd^\varphi f$.
\end{example}

Our goal is to obtain an analogue of Corollary \ref{cor:BottChern}. Let us start by defining the analogous harmonic 3-form (see also \cite{GoncaloGerbes}).

\begin{prop}\label{prop:coass}
Let $(M,\varphi)$ be a compact torsion-free $G_2$ manifold and $C$ be a compact coassociative submanifold. Let $[C]\in H_4(M;\Z)$ denote the corresponding homology class and $\mathcal{G}$ denote the corresponding gerbe. Then the harmonic $3$-form $\sigma$ in $c_1(\mathcal{G})$ belongs to $\mathcal{H}^3_1\oplus \mathcal{H}^3_{27}$
and defines an element $[\sigma]_\varphi\in H^{\varphi}(M)$. The class $[\sigma]_\varphi$ corresponds to $c_1(\mathcal{G})$ via the immersion $H^\varphi(M)\rightarrow H^3(M;\R)$. 
\end{prop}
\begin{proof}
Coassociative submanifolds can be characterized by the condition $\varphi_{|C}=0$. Recall also that $c_1(\mathcal{G})$ is the Poincar\'e dual of $[C]$. The fact that the Laplacian preserves types provides a decomposition $\sigma=\sigma^1+\sigma^7+\sigma^{27}$ into harmonic components. In particular, $\alpha\coloneqq \star\sigma^7\in\Lambda^4_7$ is closed. Recall that such forms can be written as $\alpha=\alpha'\wedge\varphi$, for some $\alpha'\in\Lambda^1$. Thus, using the coassociative condition and Poincar\'e duality,
$$0=\int_C\star\sigma^7=\int_M\star\sigma^7\wedge\sigma=\int_M|\star\sigma^7|^2\vol_g.$$
This shows that $\sigma^7=0$. Since $\sigma$ is coclosed, it belongs to $\ker(\pi^2_{14}d^*)$ so it defines a class in $H^\varphi(M)$. 
\end{proof}

\begin{remark}
Furthermore, write $\sigma=\sigma^1+\sigma^{27}=c\varphi+\sigma^{27}$. Then, using $d\psi=0$, 
$$\Vol(C)=\int_C\psi=\int_M\psi\wedge\sigma=7c\,\int_M\vol_g = 7c\,\Vol(M),$$
so $c$ is simply the factor relating the two volumes. This shows that $c$ is positive.
\end{remark}

We now want to associate a connection to this harmonic 3-form. We will rely on the following, general, result \cite{HitchinGerbes}.

\begin{prop}\label{p:Hitchin}
Let $(M,g)$ be a compact smooth oriented Riemannian manifold and $\mathcal{G}\in H^2(M;C^{\infty}(\C^*))$ be a geometric gerbe. 
Then any closed $3$-form representing $c_1(\mathcal{G})\in H^3(M;\Z)$ is the curvature of some connection on $\mathcal{G}$.

In particular, the harmonic $3$-form $\sigma\in c_1(\mathcal{G})$ is the curvature of some connection on $\mathcal{G}$.
\end{prop}

The initial choice of $g$ serves both to define harmonic forms and for the construction of this connection.

This leads to the following result, analogous to Corollary \ref{cor:BottChern}.

\begin{cor}\label{cor:G2BottChern}
Let $(M,\ph)$ be a compact torsion-free $G_2$ manifold. Let $C$ be a compact coassociative submanifold and $\mathcal{G}$ be the corresponding gerbe.
\begin{enumerate}
\item There exists a unique harmonic $3$-form $\sigma\in \mathcal{H}^3_1\oplus\mathcal{H}^3_{27}$ in the class $c_1(\mathcal{G})$. 
This form is the curvature of a connection on $\mathcal{G}$.
\item Any other element in the $dd^\varphi$-class $[\sigma]_\varphi=c_1(\mathcal{G})$ has the form $\sigma+dd^\varphi f$, for some $f$. These elements coincide with the curvatures obtained by twisting $\mathcal{G}$ with the trivial gerbe $\mathcal{G}_0$, endowed with the connection of Example \ref{e:G2gerbeconnection}.
\end{enumerate}
\end{cor}

As a final remark, notice that in Corollary \ref{cor:BottChern} the potential $f$ was incorporated into the geometry by using it to conformally change the metric on $L_0$. This does not have a counterpart in Corollary \ref{cor:G2BottChern}. We can however consider the following variation.

Let us restrict our attention to positive functions $f:M\rightarrow \R$. The choice of $f$ defines a new $G_2$ structure $\varphi_f\coloneqq f^3\varphi$.
Let $g_f=f^2g$ denote the corresponding metric and $\star_f$ the corresponding Hodge operator. 

\begin{lem}Let $\varphi$ be torsion-free. Let $\{\varphi_f\}$ denote its conformal class. Then:
\begin{enumerate}[(1)]
\item $d^{\star_f}\varphi=d^\ph(\tfrac12 f^{-2})$.\\ 
In particular, $\Delta^{g_f}\varphi=dd^{\star_f}\varphi = dd^\ph(\tfrac12 f^{-2})$.
\item A submanifold $C$ is coassociative with respect to $\varphi$ if and only if it is coassociative with respect to any $\varphi_f\in \{\varphi_f\}$.
\end{enumerate}
\end{lem}
\begin{proof}Regarding point (1), recall that $\star_f\alpha = f^{7-2k}\star\alpha$, for all $\alpha\in\Lambda^k(M)$. Then
\[
d^{\star_f}\ph	= - \star_f d\star_f \ph = - f^{-3} \star d (f\star\ph) = - f^{-3} \star (df \wedge \star\ph)
			= \star \left(d\left(\tfrac12 f^{-2}\right) \wedge \star\ph\right). 
\]
To conclude, notice that $\Delta^{g_f}\coloneqq dd^{\star_f}+d^{\star_f}d$.

Point (2) is a consequence of the fact that $C$ is coassociative if and only if $\varphi_{|C}=0$. Alternatively, let $\iota:C\hookrightarrow M$ denote the immersion. By definition, $\iota^*\psi=\vol_{\iota^*g}$. Notice that $\iota^*g_f= f^2\iota^*g$ so, since $C$ is 4-dimensional, the volume form satisfies
$\vol_{\iota^*g_f}=f^4\vol_{\iota^*g}$. On the other hand, $\psi_f=\star_f\ph_f = f^4 \star \ph$ so $\iota^*\psi_f = f^4\iota^*{\star \ph}=f^4\vol_{\iota^*g}=\vol_{\iota^*g_f}$.
\end{proof}

This provides an alternative construction of (many of) the same connections on the trivial gerbe already seen in Example \ref{e:G2gerbeconnection}.

\begin{example}\label{e:G2gerbeconnection_revised}
Let $(M,\varphi)$ be a compact torsion-free $G_2$ manifold. Let $\mathcal{G}_0$ be the trivial gerbe. Choose any global function $f$ on $M$ such that $f>0$.
Set $h \coloneqq (2f)^{-1/2}$. Then $F\coloneqq d^{\star_h}\varphi$ defines a connection on $\mathcal{G}_0$, 
with curvature $dd^{\star_h}\varphi=dd^\varphi f$.
\end{example}

In the setting of Corollary \ref{cor:G2BottChern}, we can replace Example \ref{e:G2gerbeconnection} with Example \ref{e:G2gerbeconnection_revised}. Comparing this construction with Corollary \ref{cor:BottChern}, and ignoring the restriction to positive functions, the main difference lies in the manifestation of conformality. Here, $c_1(\mathcal{G})\in H^\varphi(M)$ parametrizes the curvatures on $\mathcal{G}$ defined by a conformal class of $G_2$ structures on $M$. There, $c_1(L)\in H^{1,1}_{BC}(M)$ parametrized the curvatures on $L$ defined by a conformal class of metrics on $L_0$. The fact that, in the $G_2$ case, $f$ directly affects the structure on $M$ seems in line with the general principle that $G_2$ geometry is ``less linear" than its complex counterpart. 

\begin{remark}
Notice that the role of the metric $g$ in Proposition \ref{p:Hitchin} is rather subtle: it selects a specific connection on $\mathcal{G}$ with the desired curvature. It would be interesting to develop a better understanding of the algebraic properties, in particular the type, of connections on $\mathcal{G}$ built using metrics on $M$ induced by $G_2$ structures.
\end{remark}

\appendix

\section{Appendix: Analytic tools}\label{Sect:AnalyticTools}

We shall review here the analytic techniques used to prove Hodge decompositions. 
Although this theory is completely standard, it is perhaps more useful to collect here the specific facts used in this article,  
instead of referring to multiple papers in the literature.
At the end we shall also review the proof of the Riemannian Hodge decomposition. 
This will help illustrate these techniques in the simplest situation: elliptic operators. 
It will also help to emphasize the role of regularity results and to show how different decompositions may be interrelated. 

For our purposes it will suffice to consider real vector spaces, real vector bundles etc. 

\begin{enumerate}[1.]
\item 
Consider a linear differential operator $P:\Lambda^0(E)\rightarrow\Lambda^0(F)$ of order $m$, with smooth coefficients, 
between smooth sections of vector bundles $E,F$ on $M$. The principal symbol of $P$ is a bundle map
\[
\sigma(P):T^*M\rightarrow E^*\otimes F\simeq \mathrm{Hom}(E,F).
\]
It is said to be injective if the following condition holds: 
$\forall x\in M,\forall \xi\in T^*_xM\smallsetminus\{0\}$, the homomorphism $\sigma(P)|_{x}(\xi):E_x\rightarrow F_x$ is  injective.

\item 
If $(M,g)$ is oriented and Riemannian and $E,F$ are tensor bundles, endowed with the induced metrics and connections, 
we can define Hilbert spaces $H^k(E), H^k(F)$ of sections with $k\geq 0$ weak derivatives (defined distributionally) in $L^2$. 
By definition, $H^0(E)=L^2(E)$, $H^0(F)=L^2(F)$.

If $M$ is compact, $H^k(E)$ contains $\Lambda^0(E)$ as a dense subspace. 
More generally, let $\Lambda^0(E)'$ denote the topological dual space, \textit{i.e.}, the space of distributions. The embedding
\[
H^0(E)\rightarrow\Lambda^0(E)', \quad  e\mapsto\int_M g(e,\cdot)\vol_g
\] 
restricts to define a chain of embeddings
\[
\Lambda^0(E)<H^k(E)<\Lambda^0(E)'.
\]
The analogue holds for $F$. 

\textit{Notation. } From now on, we shall simplify the notation as follows
\[
\int_M(\cdot,\cdot) \coloneqq \int_M g(\cdot,\cdot)\vol_g. 
\]

Again using compactness and the metrics, we can define an adjoint operator $P^t:\Lambda^0(F)\rightarrow\Lambda^0(E)$ via integration by parts, 
\textit{i.e.}, such that, for all $e\in\Lambda^0(E),f\in\Lambda^0(F)$,
\[
\int_M(P(e),f) = \int_M(e,P^t(f)).
\]
The principal symbols of $P, P^t$ are related by the identity $\sigma(P^t)= (-1)^m\sigma(P)^t$.
 
\textit{Notation. }Below, when we wish to emphasize the role of smooth sections as test functions, we will use the notation $\eta\in\Lambda^0(E)$, 
$\phi\in\Lambda^0(F)$.

\item 
Assume $M$ is compact.  We can extend $P$ to an operator $P:H^{j+m}(E)\rightarrow H^j(F)$ in two, equivalent, ways. 
To prove their equivalence it will suffice to notice that, since $P(e)\in H^j(F)\leq L^2(F)$, it is completely characterized by the set of values 
$\{\int_M (P(e),\phi):\phi\in \Lambda^0(F)\}$. 

The simplest way is by substituting standard derivatives with weak derivatives in the definition of $P(e)$. 
We then obtain by definition that, for all $\phi\in\Lambda^0(F)$, 
$$\int_M(P(e),\phi)=\int_M(e,P^t(\phi)).$$

The second way is by continuity. 
Since $H^{j+m}(E)$ is the completion of $\Lambda^0(E)$, given $e\in H^{j+m}(E)$ there exists $e_n\in\Lambda^0(E)$ such that 
$\lim_{n\to\infty}e_n= e$. 
Convergence implies $e_n$ is a Cauchy sequence and Cauchy sequences are preserved by continuous maps, 
so $P(e_n)$ converges to some $f\in H^j(F)$. We then define $P(e) \coloneqq f$. 
Using this definition,
\[
\int_M (P(e),\phi) = \lim_{n\to\infty} \int_M (P(e_n),\phi) = \lim_{n\to\infty} \int_M(e_n,P^t(\phi)) = \int_M (e,P^t(\phi)).
\]
More generally, $P$ can be extended to an operator between distribution spaces:
\begin{equation}\label{eq:distr_extension}
P:\Lambda^0(E)'\rightarrow\Lambda^0(F)', \ \ P(e)(\phi) \coloneqq e(P^t(\phi)).
\end{equation}
If $e,P(e)\in L^2$, the distributional definition can be expressed via integration 
\[
\int_M (P(e),\phi)=P(e)(\phi)=\int_M (e,P^t(\phi)),
\] 
so this definition indeed extends the definition of $P$ on $H^{j+m}(E)$.

\item 
General theory implies that operators whose symbol is injective have the following properties: 
\begin{enumerate}[(i)] 
\item Regularity: given any section $e\in H^{j+m}(E)$, if $P(e)\in\Lambda^0(F)$ then $e\in\Lambda^0(E)$. More generally, given any distribution $e\in\Lambda^0(E)'$, if $P(e)\in\Lambda^0(F)$ 
(viewed as a subspace of $\Lambda^0(F)'$) then $e\in\Lambda^0(E)$. 
\item  On a compact manifold, the extended operator $P:H^{j+m}(E)\rightarrow H^j(F)$ is semi-Fredholm: $\Ker(P)$ has finite dimension and $\Im(P)$ is a closed subspace of $H^j(F)$. Notice that, by regularity, $\Ker(P)$ in $\Lambda^0(E)'$  coincides with $\Ker(P)$ in $\Lambda^0(E)$.
\end{enumerate}
\begin{remark}Both facts can be proved using analogous properties in the better-known case of elliptic operators. Indeed, $P^tP:H^{j+m}(E)\rightarrow H^{j-m}(E)$ is elliptic, \textit{i.e.}, its principal symbol is an isomorphism. 

For example, $P(e)\in\Lambda^0(F)$ implies $P^tP(e)\in\Lambda^0(E)$, thus $e\in\Lambda^0(E)$ by elliptic regularity. 
Analogously, $\Ker(P)\leq\Ker(P^tP)$ is finite-dimensional. 

Furthermore, standard elliptic theory implies the existence of a continuous operator $A:H^{j-m}(E)\rightarrow H^{j+m}(E)$ such that 
\[
I-A(P^tP)=K:H^{j+m}(E)\rightarrow H^{j+m}(E),
\]
where $K$ is a compact operator. We will show that $\Im(P)$ is closed. 

Let $f\in \overline{\Im(P)}$. We may assume $f\neq0$.  
Since $\Im(P)=\Im(P|_{\Ker(P)^\perp})$, there exists a sequence $\{e_n\}\subset\Ker(P)^\perp$ such that $P(e_n)\rightarrow f$. 
This sequence is bounded. Indeed, up to subsequences, if $\|e_n\|\rightarrow\infty$ then 
$e_n/\|e_n\|=(1/\|e_n\|)AP^tP(e_n)+K(e_n/\|e_n\|)$ 
would converge to some $e\in\Ker(P)^\perp$. 
By continuity of the norm, $\| e\|=1$ so $e\neq0$. 
Again by continuity, $\|P(e)\|=\lim((1/\|e_n\|)\|P(e_n)\|)= 0$ so $P(e)=0$. This contradicts $e\in\Ker(P)^\perp\smallsetminus\{0\}$.  
We may thus use the identity $e_n=AP^tP(e_n)+K(e_n)$ to prove that, up to a subsequence, $e_n$ converges to some $e\in H^{j+m}(E)$. 
Continuity implies that $P(e)=f$, proving that $\Im(P)$ is closed.
\end{remark}

\end{enumerate}

In this article we shall apply these facts to the case where $P$ is a $\phi$-Hessian operator (cf.~Def.~\ref{Def:phiHess}) with $m=2$. 
In each specific case of interest, we will compute the symbol and show that it is injective so as to take advantage of the above properties. 

The Hodge decomposition for $P$ will be based on the following, well-known, fact.

\begin{prop}
Let $H$ be a Hilbert space. Let $Z\leq H$ be a closed subspace. Then there exists an orthogonal decomposition $H=Z\oplus Z^\perp$.
\end{prop}

When $P$ has injective symbol, we can apply this to the map $P:H^2(E)\rightarrow H^0(F)=L^2(F)$, obtaining a decomposition 
\[
L^2(F) = \Im(P)\oplus \Im(P)^\perp.
\]

The space $\Im(P)^\perp$ is rather abstract. We can try to improve this decomposition by introducing a map $P^*$ such that $\Im(P)^\perp=\Ker(P^*)$. 
We will rely on the framework provided by the theory of unbounded operators between Hilbert spaces. Let us recall a few more facts.

\begin{enumerate}[1.]
\item[5.]
Fix a linear map $T:H_1\to H_2$ between Hilbert spaces. In general, $T$ might be unbounded, equivalently non-continuous. 
In this setting the domain of definition of $T$ is usually only a subspace $D\leq H_1$. By restriction many such domains are possible   
so, as usual, the choice of a specific domain of definition is an important part of the data which defines $T$. 

\begin{example}
Let $M$ be compact and $T=P$ be a second-order linear differential operator with injective principal symbol, as above. Set $H_1 \coloneqq L^2(E)$, $H_2 \coloneqq L^2(F)$. 
In this context $T$ is generally not continuous. Two natural domains for $T$ are $H^2(E)$ and $\Lambda^0(E)$. Both are dense in $H_1$. 
If the principal symbol of $T$ is injective and we choose the domain $H^2(E)$, the image of $T$ is closed in $H_2$.
\end{example} 

\item[6.]
If $D$ is dense in $H_1$, we can define the adjoint map $T^*$ in two steps, as follows.

Step 1. The domain $D'$ of $T^*$ is the set of $w\in H_2$ such that the corresponding component of $T$
\[
D\rightarrow\R, \quad v\mapsto (T(v),w),
\]
is continuous. It thus extends by continuity to a linear map $w_T:H_1\rightarrow\R$.

Step 2. Given $w\in D'$, the Riesz representation theorem implies that there exists a unique element in $H_1$, which we denote $T^*(w)$, 
such that $w_T(\cdot)=(\cdot,T^*(w))$. Restricting to $D$ we find the following relation:
\[
(T(v),w)=(v,T^*(w)).
\]
One can check that the map $w\in D'\mapsto T^*(w)\in H_1$ is linear.

The bottom line is that we have associated, to each pair $(T,D)$ with $D$ dense in $H_1$, a pair $(T^*,D')$. 
The map $T^*$ is generally not continuous, and its domain $D'$ may or may not be dense in $H_2$. 

As mentioned, the reason we are interested in $T^*$ is the following.

\begin{lem}
Given $(T,D)$ with $D$ dense in $H_1$, $\Im(T)^\perp=\ker(T^*)$. 
\end{lem}
\begin{proof}
Let $w\in D'$ such that $T^*(w)=0$. Then, $\forall v\in D$, $0=(v,T^*(w))=(T(v),w)$, so $w\in \Im(T)^\perp$.

Conversely, assume $w\in \Im(T)^\perp$. The map $v\in D\mapsto (Tv,w)=0$ is continuous so $w\in D'$. Clearly, $w_T\equiv 0$ so $T^*(w)=0$. 
\end{proof}

\item[7.] 
Applying the above to the case $T = P$ leads us to the following decomposition.

\begin{prop}\label{p:abstractdecomp}
Let $M$ be an oriented compact Riemannian manifold and $P$ be a second-order linear differential operator between tensor bundles $E,F$, 
with injective principal symbol. Set $H_1 \coloneqq L^2(E)$ and choose $D\coloneqq H^2(E)$. Set $H_2 \coloneqq L^2(F)$ 
and let $(P^*,D')$ be the adjoint operator. Then 
\[
L^2(F)=\Im(P)\oplus \Ker(P^*).
\]
\end{prop}

The definition of $P^*$ is still very abstract. Our final goal is to relate it to the more explicitly defined operator $P^t$.

As in Point 3.~above, the operator $P^t:\Lambda^0(F)\rightarrow \Lambda^0(E)$ extends to an operator
\[
P^t:\Lambda^0(F)'\rightarrow\Lambda^0(E)', \quad P^t(f)(\eta) \coloneqq f(P(\eta)).
\]
Assume $f$ lies in the domain of $P^*$. Then $f\in L^2(F)$ and $P^*(f)\in L^2(E)$, so
\[
P^t(f)(\eta) = f(P(\eta)) = \int_M(f,P(\eta)) = \int_M (P^*(f),\eta) = P^*(f)(\eta).
\]
This means that $P^t(f)=P^*(f)$, \textit{i.e.}, $P^*$ is a restriction of the distributional operator. 
Notice also that any $f\in\Lambda^0(F)$ belongs to the domain of $P^*$. Indeed, the map
\[
H^2(E)\rightarrow\R, \quad e \mapsto \int_M(P(e),f) = \int_M(e,P^t(f))
\]
is continuous, so it extends to a map $L^2(E)\rightarrow\R$.

We may conclude that $P^*$ is a distributional extension of the original operator $P^t:\Lambda^0(F)\rightarrow \Lambda^0(E)$. 
\end{enumerate}

\paragraph{Application.}The simplest application of the above theory is as follows. 

Assume $P$ in Proposition $\ref{p:abstractdecomp}$ is elliptic, \textit{i.e.}, its principal symbol is an isomorphism. 
This implies that $\mathrm{rank}(E)=\mathrm{rank}(F)$ and that $\sigma({P^t})$ is injective, so we can apply the regularity property to $P^*$ finding $\ker(P^*)=\ker(P^t)\leq\Lambda^0(F)$. 
This leads to the decomposition $L^2(F)=\Im(P)\oplus \Ker(P^t)$. 

\medskip

The primary example of this situation is as follows.

\begin{thm}[Riemannian Hodge decomposition]\label{thm:RiemannianHodge}
Let $(M,g)$ be a compact oriented Riemannian manifold. Let $\Delta\coloneqq dd^*+d^*d:\Lambda^k(M)\rightarrow\Lambda^k(M)$ be the Hodge Laplacian operator on k-forms. Then $\Lambda^k(M)=\Im(\Delta)\oplus\ker(\Delta)$.

Furthermore, $\ker(\Delta)$ is finite-dimensional and the decomposition is orthogonal with respect to the natural $L^2$-metric on $\Lambda^k(M)$.
\end{thm}

\begin{proof}
The Hodge Laplacian is elliptic and self-adjoint, \textit{i.e.}, $\Delta^t=\Delta$, so 
\[
L^2(\Lambda^k(M))=\Im(\Delta)\oplus\Ker(\Delta).
\] 
Regularity implies that $\Ker(\Delta)\leq\Lambda^k(M)$. If we intersect both sides by $\Lambda^k(M)$, we thus obtain
$$\Lambda^k(M)=(\Im(\Delta)\cap\Lambda^k(M))\oplus\Ker(\Delta).$$
Here, $\Im(\Delta)$ still refers to the extended operator. However, regularity shows that $\Im(\Delta)\cap\Lambda^k(M)$ coincides with the image of the original operator $\Delta:\Lambda^k(M)\rightarrow\Lambda^k(M)$. We may thus reformulate the above as follows:
\begin{equation}\label{eq:RiemannianHodge}
\Lambda^k(M)=\Im(\Delta)\oplus\Ker(\Delta).
\end{equation}
This statement is completely expressed in terms of $\Delta$ on smooth forms, so we have reached the desired conclusion.
\end{proof}

The operator $d:\Lambda^{k-1}(M)\rightarrow\Ker(d)\leq \Lambda^k(M)$ is more difficult to deal with because its principal symbol is not injective. We can however bypass this issue in two ways. 

One way starts with the observation that $\Im(\Delta)\subseteq \Im(d)\oplus \Im(d^*)$. Simple linear algebra (orthogonality arguments) and integration by parts, applied to equation \ref{eq:RiemannianHodge}, shows that the converse also holds. This leads to the orthogonal splitting
\begin{equation}\label{eq:RiemannianHodge_bis}
\Lambda^k(M)=\Im(d)\oplus\Ker(\Delta)\oplus\Im(d^*).
\end{equation}
The first two terms on the RHS can be identified with $\Ker(d)$, leading to a Hodge decomposition for $d$:
\begin{equation}\label{eq:dHodge}
\Ker(d)=\Im(d)\oplus\Ker(\Delta).
\end{equation}
The other way is to notice that the operator $d+d^*$ is elliptic and self-adjoint. As in Theorem \ref{thm:RiemannianHodge}, we thus get a splitting
$$\Lambda^*(M)=\Im(d+d^*)\oplus \Ker(d+d^*).$$
Linear algebraic arguments allow us to restrict this to the subspace $\Lambda^k(M)$, leading back to equation \ref{eq:RiemannianHodge_bis}.

We can formalize equation \ref{eq:dHodge} as follows. Recall the definition of de Rham cohomology: given the complex
$$\cdots\rightarrow\Lambda^{k-1}(M)\xrightarrow[]{d}\Lambda^{k}(M)\xrightarrow[]{d}\Lambda^{k+1}(M)\rightarrow\cdots,$$
one sets $H^k(M;\R)\coloneqq \ker(d)/\Im(d)$. 
The above results show that:
\begin{enumerate}[(i)]
\item $H^k(M;\R)\simeq\Ker(\Delta)$. In particular, it is finite-dimensional. 
\item Any exact form $\alpha\in\Lambda^k(M)$ is ``doubly exact'', in the following sense: $\alpha=d\beta$ implies $\alpha=dd^*\alpha'$, for some $\alpha'\in \ker(d)$. 
\item This $\alpha'$ is uniquely defined up to harmonic forms, so the choice $\beta\coloneqq d^*\alpha'$ provides a canonical solution to the equation $d\beta=\alpha$. 
\end{enumerate}

\begin{remark}
The fact that exactness implies double exactness is vaguely similar to the global $\partial\bar\partial$ lemma for compact K\"ahler manifolds. However, the $\partial\bar\partial$ operator does not depend on the metric and the corresponding lemma is a less elementary consequence of the Laplacian Hodge decomposition. 
\end{remark}

\bibliographystyle{amsplain}
\bibliography{ddphi_biblio}

\end{document}